\documentclass{article}
\usepackage[superscript,biblabel]{cite}
\usepackage[utf8]{inputenc}
\usepackage[affil-it]{authblk}
\usepackage{amsmath}
\usepackage{amsbsy}
\usepackage{amssymb}
\usepackage{graphicx}
\usepackage{dsfont}
\usepackage{upgreek}
\usepackage{textcomp}
\usepackage{braket}
\usepackage{setspace}
\usepackage{blindtext} 
\usepackage{blindtext} 
\usepackage[margin=1in]{geometry}
\usepackage{hyperref}
\hypersetup{
    colorlinks=true, 
    linktoc=all,     
    linkcolor=blue,  
}
\usepackage{tocloft} 
\usepackage{amsthm}
\usepackage{mathrsfs}
\usepackage{mathtools}
\usepackage[table,svgnames]{xcolor}
\usepackage{graphicx}
\usepackage{tikz}
\usepackage{float}
\usepackage{enumerate}
\usetikzlibrary{arrows}
\usepackage[toc,page]{appendix}
\usepackage{cleveref}
\usepackage{xcolor}

\crefname{appsec}{appendix}{appendices}
\numberwithin{equation}{section}

\newtheorem{theorem}{Theorem}[section]
\newtheorem{lemma}{Lemma}[section]
\newtheorem{corollary}{Corollary}[section]
\newtheorem{remark}{Remark}[section]
\newtheorem{definition}{Definition}[section]
\newtheorem{proposition}{Proposition}[section]

\providecommand{\keywords}[1]
{
  \textbf{Key Words and Phrases.} #1
}
\providecommand{\classification}[1]
{
  \textbf{Mathematics Subject Classification.} #1
}
\title{Well-posedness of the mixed-Fractional Nonlinear Schr\"odinger Equation on $\mathbb{R}^2$}

\author{Brian Choi\thanks{Corresponding author. Contact: \texttt{choigh@smu.edu}}, Alejandro Aceves\thanks{Contact: \texttt{aaceves@smu.edu}}%
  }
\affil{Department of Mathematics, Southern Methodist University, Dallas, TX 75275, USA}
\date{}
\begin{document}
\maketitle\vspace{-8ex}
\begin{abstract}
We investigate the well-posedness theory of the 2-D fractional nonlinear Schr\"odinger equation (NLSE) with a mixed degree of derivatives. Motivated by models in optics and photonics where the light propagation is governed by non-quadratic, fractional, and anisotropic dispersion profile, this paper presents first results in this direction. Dispersive estimates are developed in the context of anisotropic Sobolev spaces defined by inhomogeneous symbols. The main model is shown to exhibit scattering for small data in the scaling-critical space. Furthermore the continuity of solution with respect to the dispersion parameter is shown on a compact time interval.
\end{abstract}
\classification{35B30, 35Q40, 35Q55, 35Q60,35R11}\\
\begin{spacing}{0.001}
\tableofcontents
\end{spacing}
\singlespacing
\section{Introduction.}

This paper is concerned with the well-posedness and regularity properties of the mixed fractional nonlinear Schr\"odinger equation (mNLSE)
\begin{equation}\label{mixedNLS}
    i\partial_t u = (D_1^{\alpha_1}+ D_2^{\alpha_2})u + \mu |u|^{p-1}u,\:u(x,y,0) = u_0(x,y),\:(x,y,t) \in \mathbb{R}^2 \times \mathbb{R},
\end{equation}
where $D_1 = (-\partial_{xx})^{\frac{1}{2}},\: D_2 = (-\partial_{yy})^{\frac{1}{2}}$ and $\mu = \pm 1,\: p >1,\: \alpha_1,\alpha_2 \in (0,2]\setminus \{1\},\: \alpha_1 \geq \alpha_2$.

Interest in this model comes from the field of optics and photonics \cite{austin, longhi}, where the fractional operator accounts for engineering spatial diffraction in an optical cavity ($\alpha_2 < 2$) \cite{longhi}, while the second order operator models chromatic dispersion of optical pulses ($\alpha_1 = 2$). Another interesting configuration is that of an array of resonators globally coupled \cite{kirpatrick}, for which optical pulses in each resonator are modeled by the one-dimensional Schr\"odinger equation. In this case, (1.1) represents the continuum approximation of such system.

The well-posedness theory of the fractional NLSE (where $\alpha_1=\alpha_2$) has been investigated by several authors. Our approach is based on that of \cite{1534-0392_2015_6_2265,dinh:hal-01426761} where the contraction mapping argument is developed based on the Strichartz estimates corresponding to $e^{-it(-\Delta)^{\alpha/2}}$. For another approach based on Bourgain's method of restricted norm, see \cite{1078-0947_2015_7_2863}. When $\alpha=1$, the non-dispersive solutions and their blow-up criteria have been studied in \cite{krieger2013nondispersive}.

The presence of mixed derivatives in \cref{mixedNLS} necessitates a non-trivial modification to the standard fixed point argument in solving NLSE, if one wishes to obtain the well-posedness theory at the scaling-critical regularity. One major component of this paper is the analysis of the non-smooth Littlewood-Paley decomposition based on a dyadic family of non-smooth symbols. Furthermore this functional framework gives rise to the use of anisotropic Sobolev spaces.

When $\alpha_1=\alpha_2=2$, \cref{mixedNLS} reduces to the classical NLSE whose various properties are summarized in \cite{tao2006nonlinear,cazenave2003semilinear}. It is well known that the NLSE is ubiquitous in the field of nonlinear waves. It arises as the governing equation of wave-packets in deep water waves, or the electric field envelope of an intense light filament propagating in the atmosphere or it models pulse propagation in an optical fiber. In condensed matter physics this equation known as the Gross–Pitaevskii equation (GPE), describes the dynamics of Bose-Einstein condensates (BEC). The underlying competing effects embedded in the equation are linear dispersive (diffractive) properties of the media, best described in the Fourier space as a frequency, wave-number quadratic relation $\omega = |k|^2$ or as the Laplacian operator in real-space, and nonlinear effects which in the optics application represents intensity dependent index of refraction and in the BEC comes from a postulate that many-body effects can be
compressed into a nonlinear on-site interaction. Alternatively, the discrete NLSE, where  in a 1-dimensional array, the second derivative term is replaced by the well-known center difference scheme 
$\frac{\partial^2U(t,x)}{\partial x^2} \rightarrow \Big(u_{n+1}(t) -2u_n(t) + u_{n+1}(t)\Big)$ has equally important applications. In the field of photonics for example, this discrete operator describes nearest neighbor interactions of optical pulses propagating in waveguide arrays. An important assumption of linear operators such as the Laplacian, is that the medium is considered to be homogeneous. In the discrete case, the physics equivalent assumption is that interactions are only between nearest neighbors (local). In this work we will depart from this assumption of homogeneity and the quadratic dispersion 
profile.

Beyond the classical interpretation of the Laplacian in many physical systems, from the stochastic process perspective (see \cite{oksendal2013stochastic}), the Laplacian is the infinitesimal generator of the Wiener process, which constitutes a special case of the more general L\'{e}vy process characterized by the L\'{e}vy index $\alpha \in (0,2]$, which in turn, generates the fractional Laplacian $(-\Delta)^{\frac{\alpha}{2}} = \mathcal{F}^{-1}|\xi|^\alpha \mathcal{F}$. The fractional NLSE was considered by Laskin (see \cite{laskin2000fractional,laskin2002fractional}) in an attempt to extend the Feynmann path integral over the Brownian-like to the L\'{e}vy-like paths. Whereas Laskin considered $\alpha \in (1,2]$ for the sake of physical applications (for instance, see the discussion on the energy spectrum of a hydrogenlike atom in \cite{laskin2002fractional}), our analysis contains $\alpha_1 = \alpha_2 \in (0,2]\setminus \{1\}$.

This paper is summarized as follows. In \cref{result}, the main results are stated. In \cref{function spaces,wp}, various linear and nonlinear estimates are developed with which the well-posedness proofs are obtained in \cref{gwp}. In \cref{gwp}, global well-posedness for data with finite energy is also discussed. In \cref{alpha regularity}, the regularity of solutions with respect to dispersive parameters ($\alpha_i$) is discussed. Some technical results are contained in the Appendix. 
\section{Main Results.}\label{result}

The model \cref{mixedNLS} is a Hamiltonian PDE with two notable conserved quantities:
\begin{equation}\label{conservation}
\begin{split}
\text{Mass}:\: M[u(t)] &= \iint |u(x,y,t)|^2 dxdy\\
\text{Energy}:\: E[u(t)] &= \iint \frac{1}{2}\Big(|D_1^{\frac{\alpha_1}{2}} u(x,y,t)|^2 + |D_2^{\frac{\alpha_2}{2}}u(x,y,t)|^2\Big) +\frac{\mu}{p+1} |u(x,y,t)|^{p+1} dxdy.
\end{split}
\end{equation}
If $u$ is a classical solution, then so is
\begin{equation}\label{scaling}
    u_\lambda(x,y,t) \coloneqq \lambda^{-\frac{\alpha_1}{p-1}} u\Big(\frac{x}{\lambda},\frac{y}{\lambda^{\frac{\alpha_1}{\alpha_2}}},\frac{t}{\lambda^{\alpha_1}}\Big)
\end{equation}
for every $\lambda>0$; when $\alpha_1 \neq \alpha_2$, the two space variables $x,y$ do not obey the same scaling law. This motivates us to consider the class of data in an anisotropic Sobolev space. Define 
\begin{equation*}
\| f \|_{H^s_{\alpha}(\mathbb{R}^2)} = \| (1+|\xi|^2 + |\eta|^{\alpha})^{\frac{s}{2}} \widehat{f} \|_{L^2_{\xi,\eta}},\: \| f \|_{\dot{H}^s_{\alpha}(\mathbb{R}^2)} = \| (|\xi|^2 + |\eta|^{\alpha})^{\frac{s}{2}} \widehat{f} \|_{L^2_{\xi,\eta}} 
\end{equation*}
where $\xi,\eta$ are the dual variables to $x,y$, respectively, and $\mathcal{F}f = \widehat{f}$ is the Fourier transform of $f$. It is assumed that all function spaces are defined on $\mathbb{R}^2$ unless specified otherwise. By construction, observe that 
\begin{equation*}
    \| u_\lambda \|_{\dot{H}_\alpha^{s}} = \lambda^{s_c-s}\| u\|_{\dot{H}_\alpha^{s}},
\end{equation*}
where $\alpha = \frac{2\alpha_2}{\alpha_1}$ and $s_c = s_c(\alpha_1,\alpha_2) := \frac{1}{2}+\frac{\alpha_1}{2\alpha_2} - \frac{\alpha_1}{p-1}$. We show that \cref{mixedNLS} is well-posed in $H^s_\alpha$ for $s \geq s_c$, i.e., in the scaling subcritical and critical regularities. Since \cref{mixedNLS} admits the time-reversal symmetry $u(x,y,t)\mapsto \overline{u(-x,-y,-t)}$, our analysis is restricted to $[0,T]$. 
\begin{theorem}\label{thm1}
Suppose $s \in (s_c, \lfloor p \rfloor-1]$ for $p \geq 3$ not an odd integer and $s \in (s_c,\infty)$ for $p \geq 3$ an odd integer. Then, \cref{mixedNLS} is locally well-posed in $H_{\alpha}^{s}$.
\end{theorem}
\begin{theorem}\label{thm2}
Let $p>3$. Further assume $s_c \leq \lfloor p \rfloor-1$ if $p$ is not an odd integer. Then, \cref{mixedNLS} is locally well-posed in $H_\alpha^{s_c}$.  Furthermore there exists $\delta=\delta(p,\alpha_1,\alpha_2)>0$ such that whenever $\| u_0 \|_{H^{s_c}_\alpha}<\delta$, there exists a unique $u_{\pm} \in H^{s_c}_\alpha$ such that
\begin{equation}\label{scattering}
    \lim\limits_{t\to\pm\infty}\|u(t) - U(t)u_{\pm} \|_{H^{s_c}_\alpha}=0.
\end{equation}
\end{theorem}
On $[0,T]$, the solution map, if it exists, is not only continuous in the variations in data but also in dispersive parameters. We consider the convergence of solutions as $\alpha_i^\prime \rightarrow \alpha_i,\: i=1,2$ where $\alpha_i^\prime \in (0,2]\setminus \{1\}$ and $\alpha^\prime = \frac{2\alpha_2^\prime}{\alpha_1^\prime}$. However this result cannot be extended to $T=\infty$, which is discussed in \cref{alpha regularity}.   
\begin{theorem}\label{limit local}
Let $p \geq 3,\: T \in (0,\infty)$, and $\frac{1}{2}+\frac{1}{\alpha}<s \leq \lfloor p \rfloor -1$. If $u^\alpha,u^{\alpha^\prime} \in C([0,T];H^s_\alpha)$ satisfy
\begin{equation}\label{continuity}
\begin{split}
    i \partial_t u^\alpha &= (D_1^{\alpha_1} + D_2^{\alpha_2})u^\alpha + \mu |u^\alpha|^{p-1}u^\alpha,\: u^\alpha(0) = u_{0,\alpha} \in H^s_\alpha \\
    i \partial_t u^{\alpha^\prime} &= (D_1^{\alpha_1^\prime} + D_2^{\alpha_2^\prime})u^{\alpha^\prime} + \mu |u^{\alpha^\prime}|^{p-1}u^{\alpha^\prime},\: u^{\alpha^\prime}(0) = u_{0,\alpha^\prime} \in H^s_\alpha 
\end{split}
\end{equation}
where $u_{0,\alpha^\prime}\xrightarrow[]{H^s_\alpha} u_{0,\alpha}$ as $(\alpha_1^\prime,\alpha_2^\prime) \to (\alpha_1,\alpha_2)$ and 
\begin{equation}\label{uniform bound}
\sup\limits_{(\alpha_1^\prime,\alpha_2^\prime):|\alpha_1-\alpha_1^\prime|+|\alpha_2-\alpha_2^\prime|<R}\| u^{\alpha^\prime} \|_{L^{p-1}_{t} L^\infty([0,T]\times\mathbb{R}^2)} \leq c \| u^\alpha \|_{L^{p-1}_{t} L^\infty([0,T]\times \mathbb{R}^2)}<\infty   
\end{equation}
for some $c>0$ and $R = R(\alpha_1,\alpha_2)>0$, then $u^{\alpha^\prime}\xrightarrow[]{}u^\alpha$ in $C([0,T];H^s_\alpha)$ as $(\alpha_1^\prime,\alpha_2^\prime) \to (\alpha_1,\alpha_2)$.
\end{theorem}
\section{Function Spaces and Strichartz Estimates.}\label{function spaces}
This section presents various linear estimates to solve \cref{mixedNLS} with data in low-regularity function spaces that are compatible with the anisotropic scaling \cref{scaling}. All implicit constants may depend on $\alpha_1,\alpha_2$. Let $\beta_i = 1-\frac{\alpha_i}{2},\: i=1,2$ throughout this paper.
\begin{definition}
Let $s\in \mathbb{R},\: p\in [1,\infty]$. Define the inhomogeneous and homogeneous derivative operators by $\langle \nabla_\alpha \rangle^s = \mathcal{F}^{-1}(1+\xi^2 + |\eta|^\alpha)^{\frac{s}{2}}\mathcal{F}$ and $| \nabla_\alpha |^s = \mathcal{F}^{-1}(\xi^2 + |\eta|^\alpha)^{\frac{s}{2}}\mathcal{F}$, respectively. Define 
\begin{equation*}
\begin{split}
    W_\alpha^{s,p}(\mathbb{R}^2) &= \{f \in \mathcal{S}^\prime:\langle \nabla_\alpha \rangle^{s}f\in L^p\},\:\dot{W}_\alpha^{s,p}(\mathbb{R}^2) = \{f \in \mathcal{S}^\prime/\mathcal{P}:| \nabla_\alpha |^{s}f\in L^p\}.
\end{split}
\end{equation*}
where $\| f \|_{W_\alpha^{s,p}} \coloneqq \| \langle \nabla_\alpha \rangle^s f \|_{L^p},\: \| f \|_{\dot{W}_\alpha^{s,p}} \coloneqq \| | \nabla_\alpha |^s f \|_{L^p}$.
\end{definition}
As usual, we denote $W_\alpha^{s,2} = H_\alpha^{s}$. By the Plancherel Theorem, it is evident that $H_\alpha^{s}$, for $s\geq 0$, defines a Hilbert space under $\langle f,g \rangle = \iint \widehat{f}(\xi,\eta)\overline{\widehat{g}(\xi,\eta)}(1+\xi^2 + |\eta|^\alpha)^s d\xi d\eta$; for $s \geq 0,\:r\in [1,\infty]$, $W_\alpha^{s,r}(\mathbb{R}^2)$ is also complete (see \cref{BesselPotential}). When $s<0,\:r\in (1,\infty)$, $W_\alpha^{s,r}(\mathbb{R}^2)$ coincides with the dual of $W_\alpha^{-s,r^\prime}(\mathbb{R}^2)$. Moreover, note the inclusion: if $s \geq 0$, then $H^s \hookrightarrow H_\alpha^{s}\hookrightarrow H^{\frac{\alpha}{2}s}$, and the inclusion reverses for $s<0$ where $W^{s,p}$ denotes the classical Sobolev space corresponding to $\alpha=2$.

A set-up of the anisotropic localization of the Fourier space is as follows. Let $\psi \in C^\infty_c((-2,2);[0,1])$ be an even function where $\psi = 1$ for $\xi \in [-1,1]$. Let $\phi(\xi)\coloneqq \psi(\xi)-\psi(2\xi)$; note that $supp(\phi)\subseteq [-2,-\frac{1}{2}]\cup [\frac{1}{2},2]$. Define 
\begin{equation}\label{LP2}
    \begin{split}
    \phi_N(\xi,\eta) &= \phi\Big(\frac{\sqrt{\xi^2+|\eta|^\alpha}}{N}\Big),\: N \in 2^\mathbb{Z}\\
    \phi^{(i)}_{N_i}(\xi,\eta) &= \phi\Big(\frac{|\xi_i|}{N_i}\Big),\: N_i \in 2^\mathbb{Z}, 
    \end{split}
\end{equation}
for $i=1,2$ where $(\xi_1,\xi_2) = (\xi,\eta)$. Define the corresponding operators 
\begin{equation}\label{LP3}
    \begin{split}
    P_N &\coloneqq \mathcal{F}^{-1}\phi_N \mathcal{F},\:P_{\leq N}  \coloneqq \mathcal{F}^{-1}\sum\limits_{N^\prime\leq  N} \phi_{N^\prime}\mathcal{F},\: P_{\sim N} \coloneqq \sum_{\frac{N}{2}\leq N^\prime \leq 2N} P_{N^\prime},
    \end{split}
\end{equation}
where $P^{(i)}_{N_i}, P^{(i)}_{\leq N_i}, P^{(i)}_{\sim N_i}$ are defined similarly. By definition of $\phi$, we have the resolution of identity
\begin{equation*}
\begin{split}
\sum\limits_{N\in 2^\mathbb{Z}} \phi_N(\xi,\eta)=1,\: \forall (\xi,\eta) \neq (0,0);\: \sum\limits_{N_i\in 2^\mathbb{Z}} \phi^{(i)}_{N_i}(\xi_1,\xi_2)=1,\: \forall \xi_i \neq 0.
\end{split}
\end{equation*}
\begin{remark}
The classical definition of (smooth) Littlewood-Paley function $\phi\Big(\frac{\sqrt{\xi^2+\eta^2}}{N}\Big)$ implies that its support is a circular annulus adapted to $\{\sqrt{\xi^2+\eta^2} \simeq N\}$; however in our case (\cref{LP2}), the support of $\phi_N$ is a deformed circle since $\alpha\neq 2$. Furthermore our definition of $\phi_N$ leads to a dyadic family of non-smooth frequency localizations with $\alpha$ as a parameter. Unless specified otherwise, we use $N,N_i$ to denote dyadic integers.
\end{remark}
Since $\phi_N$ is not smooth, $P_N$ is a convolution operator with a kernel that is not rapidly decaying. This marks a deviation from the classical case and thus various mapping properties of $P_N$ require an inspection when $\alpha \in (0,2]$, which we assume in this section.
\begin{lemma}[Bernstein's Inequality]\label{Bernstein}
For all $1 \leq p \leq q \leq \infty$,
\begin{equation}
\begin{split}
\| P_N f \|_{L^q} &\lesssim N^{(1+\frac{2}{\alpha})(\frac{1}{p}-\frac{1}{q})} \| P_N f \|_{L^p}\\
\| P_N P^{(1)}_{N_1}P^{(2)}_{N_2} f \|_{L^q} &\lesssim (N_1 N_2)^{\frac{1}{p}-\frac{1}{q}} \| P_N P^{(1)}_{N_1}P^{(2)}_{N_2} f \|_{L^p}.    
\end{split}
\end{equation}
\begin{remark}
The argument proceeds as in $\alpha=2$. See \cite{tao2006nonlinear}. Note that $supp (\phi_N) \subseteq R_1 \subseteq \mathbb{R}^2$ where $R_1$ is a rectangle of width $\sim N$ and length $\sim N^{\frac{2}{\alpha}}$. Similarly for the second inequality, the intersection of the supports (of $\phi_N, \phi^{(i)}_{N_i}$) has a volume at most $N_1N_2$ regardless of $N$.
\end{remark}
\end{lemma}
As in the classical case, the non-smooth Littlewood-Paley projections are a family of uniformly bounded operators.
\begin{lemma}\label{Bernstein5}
For $p \in [1,\infty]$, there exists $C(\alpha)>0$ independent of $p,N,N_i$ for $i=1,2$ such that
\begin{equation*}
    \| P_N u \|_{L^p}, \| P^{(i)}_{N_i} u \|_{L^p} \leq C \| u \|_{L^p}. 
\end{equation*}
\end{lemma}
\begin{proof}
To prove the first estimate, it suffices to show
\begin{equation*}
    \sup_{N\in 2^\mathbb{Z}}\|\check{\phi}_N\|_{L^1}<\infty,
\end{equation*}
and use the the Young's inequality on $P_N u = \check{\phi}_N \ast u$; the second estimate admits a simpler proof since the symbol of $P^{(i)}_{N_i}$ is smooth in $\xi,\eta$. Since
\begin{equation*}
    \check{\phi}_N(x,y) = N^{1+\frac{2}{\alpha}}\check{\phi}_1(Nx,N^{\frac{2}{\alpha}}y),
\end{equation*}
it suffices to show $\Phi:=\check{\phi}_1\in L^1(\mathbb{R}^2)$. Since $\phi_1 \in L^1(\mathbb{R}^2)$ is smooth in $\xi$ with a compact support $supp(\phi_1)\subseteq \mathbb{R}^2$ that stays away from the origin,
\begin{equation*}
\sup\limits_{x,y\in \mathbb{R}}(1+|x|^k) |\Phi(x,y)| \lesssim \| \phi_1 \|_{L^1}+\| \partial_\xi^k \phi_1 \|_{L^1} \lesssim_k 1.  
\end{equation*}
For the decay in $y$, first assume $\alpha>1$. Observing that $\partial_\eta\phi_1$ is uniformly H\"older continuous of order $\alpha-1$, we have
\begin{equation*}
\sup\limits_{x,y\in\mathbb{R}}|y|^{\alpha}|\Phi(x,y)| \lesssim_\alpha 1,    
\end{equation*}
and altogether,
\begin{equation}\label{decay1}
    |\Phi(x,y)| \lesssim_{k,\alpha} \frac{1}{1+|x|^k+|y|^{\alpha}},\:\forall k \in \mathbb{N},
\end{equation}
from which follows
\begin{equation}\label{l1}
    \| \Phi \|_{L^1} \lesssim \int_0^\infty \int_0^\infty |\Phi(x,y)| dydx \lesssim_{k,\alpha} \int_0^\infty (1+|x|)^{-k\frac{\alpha-1}{\alpha}} dx < \infty,
\end{equation}
and by taking $k\in\mathbb{N}$ sufficiently big depending on $\alpha$, it follows that $\Phi \in L^1(\mathbb{R}^2)$.

For $\alpha<1,\: \epsilon>0$, we claim
\begin{equation}\label{decay2}
    |\Phi(x,y)| \lesssim_{k,\alpha,\epsilon} \frac{1}{1+|x|^k+|y|^{1+\alpha-\epsilon}},\:\forall k \in \mathbb{N}.
\end{equation}
Note that $\partial_\eta \phi_1 (\xi,\eta)=\frac{\alpha}{2} |\eta|^{\alpha-1}sgn(\eta)\frac{\phi^\prime(\sqrt{\xi^2+|\eta|^\alpha})}{\sqrt{\xi^2+|\eta|^\alpha}}\in L^1(\mathbb{R}^2)$. By the fractional Leibniz rule, we have the pointwise estimate (in $\xi$)
\begin{equation}\label{frac leibniz}
\begin{split}
&\| D_2^{\alpha-\epsilon}\partial_\eta \phi_1 \|_{L^1_\eta}\\
&\lesssim \| D_2^{\alpha-\epsilon}(|\eta|^{\alpha-1}sgn(\eta)\chi)\|_{L^{p_1}_\eta} \|\frac{\phi^\prime(\sqrt{\xi^2+|\eta|^\alpha})}{\sqrt{\xi^2+|\eta|^\alpha}} \|_{L^{p_1^\prime}_\eta} + \| |\eta|^{\alpha-1} sgn(\eta) \chi \|_{L^{p_2}_\eta} \| D_2^{\alpha-\epsilon} \frac{\phi^\prime(\sqrt{\xi^2+|\eta|^\alpha})}{\sqrt{\xi^2+|\eta|^\alpha}} \|_{L^{p_2^\prime}_\eta}\\
&:= I + II,
\end{split}
\end{equation}
where $\chi = \chi(\eta)$ is the characteristic function on $[-2^{\frac{2}{\alpha}},2^{\frac{2}{\alpha}}]$ and $p_1,p_2 \in (1,\infty)$ are to be determined.

The second term is estimable by the Sobolev embedding. To ensure $|\cdot|^{\alpha-1}\chi(\cdot) \in L^{p_2}(\mathbb{R})$, let $p_2 \in (1,\frac{1}{1-\alpha})$. Let $\Tilde{p} \in (\frac{1}{\frac{1}{p_2^{\prime}}+1-(\alpha-\epsilon)},p_2^\prime)\cap (1,\frac{1}{1-\alpha})$ so that $W^{1,\Tilde{p}}(\mathbb{R})\hookrightarrow W^{\alpha-\epsilon,p_2^\prime}(\mathbb{R})$ and $\frac{\phi^\prime(\sqrt{\xi^2+|\eta|^\alpha})}{\sqrt{\xi^2+|\eta|^\alpha}} \in W^{1,\Tilde{p}}(\mathbb{R}_\eta)$ uniformly in $\xi$. Then,
\begin{equation*}
    II \lesssim \|\frac{\phi^\prime(\sqrt{\xi^2+|\cdot|^\alpha})}{\sqrt{\xi^2+|\cdot|^\alpha}} \|_{W^{\alpha-\epsilon,p_2^\prime}(\mathbb{R})} \lesssim \|\frac{\phi^\prime(\sqrt{\xi^2+|\cdot|^\alpha})}{\sqrt{\xi^2+|\cdot|^\alpha}} \|_{W^{1,\Tilde{p}}(\mathbb{R})} \leq C<\infty,
\end{equation*}
where $C>0$ is independent of $\xi$ since $\xi^2 + |\eta|^\alpha \simeq 1$ and $\phi^{(k)} = O_k(1)$.

To estimate $I$, we use the integral definition of the fractional Laplacian:\footnote{The numerical value of $c_{d,\alpha}=\frac{4^{\alpha/2}\Gamma(\frac{d+\alpha}{2})}{\pi^{d/2}|\Gamma(-\frac{\alpha}{2})|}$ is not used explicitly in our analysis.}
\begin{equation*}
    (-\Delta)^{\frac{\alpha}{2}}f(x) = c_{d,\alpha} \int_{\mathbb{R}^d} \frac{f(x)-f(y)}{|x-y|^{d+\alpha}}dy.
\end{equation*}
Let $f(\eta) = |\eta|^{\alpha-1} sgn(\eta)\chi(\eta)$ and $c=2^{\frac{2}{\alpha}}$. Since $f$ is odd, so is $D_2^{\alpha-\epsilon}f$, and hence assume $\eta>0$ without loss of generality. Then,
\begin{equation*}
\begin{split}
    D_2^{\alpha-\epsilon}f (\eta) &\simeq \int_{-c}^c \frac{f(\eta)-f(\eta_1)}{|\eta-\eta_1|^{1+\alpha-\epsilon}}d\eta_1\\
    &=\int_0^c \frac{\eta^{\alpha-1}\chi(\eta)-\eta_1^{\alpha-1}}{|\eta-\eta_1|^{1+\alpha-\epsilon}}d\eta_1 + \int_0^c \frac{\eta^{\alpha-1}\chi(\eta)+\eta_1^{\alpha-1}}{|\eta+\eta_1|^{1+\alpha-\epsilon}}d\eta_1=:A+B.
\end{split}
\end{equation*}
We show by direct computation the estimate:
\begin{equation}\label{fraclaplacian}
    |D_2^{\alpha-\epsilon}f(\eta)| \lesssim_{\alpha,\epsilon}
    \begin{cases} 
    \eta^{-1+\epsilon} &\mbox{if } \eta \in (0,\frac{3}{4}c]\\
    |\eta-c|^{-(\alpha-\epsilon)} & \mbox{if } \eta\in (\frac{3}{4}c,c)\cup (c,2c)\\
    \eta^{-(1+\alpha-\epsilon)} &\mbox{if } \eta \in (2c,\infty).\end{cases}
\end{equation}
If so, take $p_1=1+$ so that
\begin{equation*}
    \| D_2^{\alpha-\epsilon}(|\eta|^{\alpha-1}sgn(\eta)\chi)\|_{L^{p_1}_\eta} < \infty,
\end{equation*}
and
\begin{equation*}
    \|\frac{\phi^\prime(\sqrt{\xi^2+|\eta|^\alpha})}{\sqrt{\xi^2+|\eta|^\alpha}} \|_{L^{p_1^\prime}_\eta} \lesssim \|\frac{\phi^\prime(\sqrt{\xi^2+|\eta|^\alpha})}{\sqrt{\xi^2+|\eta|^\alpha}} \|_{L^{\infty}_\eta} \lesssim \|\frac{\phi^\prime(\sqrt{\xi^2+|\eta|^\alpha})}{\sqrt{\xi^2+|\eta|^\alpha}} \|_{L^{\infty}(\mathbb{R}^2)}<\infty.
\end{equation*}
Since the term $B$ obeys estimates similar to those of $A$, we work with $A$. By the Fundamental Theorem of Calculus, if $\eta,\eta_1 \in (a,b)$ where $a>0$, then
\begin{equation*}
    |\eta^{\alpha-1} - \eta_1^{\alpha-1}| \lesssim a^{\alpha-2}|\eta-\eta_1|.
\end{equation*}
Let $\eta \in (0,\frac{3}{4}c]$. Then,
\begin{equation*}
    \biggl|\int_0^{\eta/2} \frac{\eta^{\alpha-1}-\eta_1^{\alpha-1}}{|\eta-\eta_1|^{1+\alpha-\epsilon}}d\eta_1\biggr| \lesssim \int_0^{\eta/2} \frac{\eta_1^{\alpha-1}}{\eta^{1+\alpha-\epsilon}}d\eta_1 \lesssim \eta^{-1+\epsilon}.
\end{equation*}
For $\eta \in (0,\frac{c}{2})$,
\begin{equation*}
\begin{split}
    \biggl|\int_{\eta/2}^{2\eta} \frac{\eta^{\alpha-1}-\eta_1^{\alpha-1}}{|\eta-\eta_1|^{1+\alpha-\epsilon}}d\eta_1\biggr| &\lesssim \eta^{\alpha-2} \int_{\eta/2}^{2\eta} \frac{d\eta_1}{|\eta-\eta_1|^{\alpha-\epsilon}} \lesssim \eta^{-1+\epsilon},\\
     \biggl|\int_{2\eta}^{c} \frac{\eta^{\alpha-1}-\eta_1^{\alpha-1}}{|\eta-\eta_1|^{1+\alpha-\epsilon}}d\eta_1\biggr| &\lesssim \int_{2\eta}^c \frac{\eta^{\alpha-1}}{|\eta-\eta_1|^{1+\alpha-\epsilon}}d\eta_1 \lesssim \eta^{-1+\epsilon}.
\end{split}
\end{equation*}
Similarly for $\eta \in [\frac{c}{2},\frac{3}{4}c]$,
\begin{equation*}
    \biggl|\int_{\eta/2}^{c} \frac{\eta^{\alpha-1}-\eta_1^{\alpha-1}}{|\eta-\eta_1|^{1+\alpha-\epsilon}}d\eta_1\biggr| \lesssim \eta^{\alpha-2} \int_{\eta/2}^{c} \frac{d\eta_1}{|\eta-\eta_1|^{\alpha-\epsilon}} \lesssim \eta^{-1+\epsilon}.
\end{equation*}
Let $\eta \in (\frac{3}{4}c,c)$. Then,
\begin{equation*}
\begin{split}
    \biggl|\int_0^{\eta/2} \frac{\eta^{\alpha-1}-\eta_1^{\alpha-1}}{|\eta-\eta_1|^{1+\alpha-\epsilon}}d\eta_1\biggr| &\lesssim \int_0^{\eta/2} \frac{\eta_1^{\alpha-1}}{\eta^{1+\alpha-\epsilon}}d\eta_1 \lesssim \eta^{-1+\epsilon}\simeq_c 1\\
    \biggl|\int_{\eta/2}^{2\eta-c} \frac{\eta^{\alpha-1}-\eta_1^{\alpha-1}}{|\eta-\eta_1|^{1+\alpha-\epsilon}}d\eta_1\biggr|&\lesssim \int_{\eta/2}^{2\eta-c} \frac{\eta^{\alpha-1}}{|\eta-\eta_1|^{1+\alpha-\epsilon}}d\eta_1 \lesssim_c |\eta-c|^{-(\alpha-\epsilon)}\\
    \biggl|\int_{2\eta-c}^{c} \frac{\eta^{\alpha-1}-\eta_1^{\alpha-1}}{|\eta-\eta_1|^{1+\alpha-\epsilon}}d\eta_1\biggr|&\lesssim \int_{2\eta-c}^c \frac{(2\eta-c)^{\alpha-2}}{|\eta-\eta_1|^{\alpha-\epsilon}}d\eta_1 \lesssim_c 1+ |\eta-c|^{1-(\alpha-\epsilon)}.
\end{split}
\end{equation*}
Let $\eta \in (c,2c)$. Then,
\begin{equation*}
\begin{split}
    |A|&\leq \biggl|\int_{0}^{c/2} \frac{\eta_1^{\alpha-1}}{|\eta-\eta_1|^{1+\alpha-\epsilon}}d\eta_1\biggr|+ \biggl|\int_{c/2}^{c} \frac{\eta_1^{\alpha-1}}{|\eta-\eta_1|^{1+\alpha-\epsilon}}d\eta_1\biggr|\\
    &\lesssim \int_0^{c/2} \frac{\eta_1^{\alpha-1}}{c^{1+\alpha-\epsilon}}d\eta_1 + c^{\alpha-1} \int_{c/2}^c \frac{d\eta_1}{|\eta-\eta_1|^{1+\alpha-\epsilon}} \lesssim_c 1+ |\eta-c|^{-(\alpha-\epsilon)}.
\end{split}
\end{equation*}
Finally let $\eta \in (2c,\infty)$. Since $|\eta-\eta_1| \geq \frac{\eta}{2}$,
\begin{equation*}
    |A|\lesssim \eta^{-(1+\alpha-\epsilon)}\int_0^c \eta_1^{\alpha-1} d\eta_1 \lesssim \eta^{-(1+\alpha-\epsilon)},
\end{equation*}
and thus \cref{fraclaplacian} has been shown. Arguing as \cref{l1}, $\Phi \in L^1(\mathbb{R}^2)$ for $\alpha \in (0,1)$.

For $\alpha=1$, it follows that
\begin{equation}\label{decay3}
    |\Phi(x,y)| \lesssim_{k,\epsilon} \frac{1}{1+|x|^k+|y|^{2-\epsilon}},\:\forall k \in \mathbb{N},
\end{equation}
by arguing via the fractional Leibniz rule as \cref{frac leibniz}. In particular, the estimation of $I$ in \cref{frac leibniz} admits a simpler proof since $|\eta|^{\alpha-1} = 1$ contains no singularity at the origin.
\end{proof}
\begin{remark}\label{hypergeometric}
In \cref{frac leibniz}, the expression $D_2^{\alpha-\epsilon}(|\eta|^{\alpha-1}sgn(\eta)\chi)$ can be expressed in a closed form using the generalized hypergeometric functions. However our presentation based on direct estimation provides a more flexible approach. 
\end{remark}
The action of anisotropic derivatives on our dyadic decomposition behaves like multipliers.
\begin{lemma}\label{Bernstein3}
Let $s \in \mathbb{R}$, $r \in [1,\infty]$. Then,
\begin{equation*}
    \begin{split}
    \| |\nabla_\alpha|^{s} P_N u \|_{L^r} \simeq N^s \| P_N u \|_{L^r}, \:\| D_i^s P^{(i)}_{N_i} u \|_{L^r} \simeq N_i^s \| P^{(i)}_{N_i} u\|_{L^r},
    \end{split}
\end{equation*}
where the implicit constants are independent of $r,M,N$.
\end{lemma}
An application of the contraction mapping argument depends on the completeness of $W^{s,r}_\alpha$, which is a consequence of the boundedness of the anisotropic Bessel potential on Lebesgue spaces, similar to the classical case. The proofs of \cref{Bernstein3,BesselPotential} are in the Appendix. 
\begin{lemma}\label{BesselPotential}
For any $s\geq 0$, $r \in [1,\infty]$, $\alpha \in (0,2]$,
\begin{equation}\label{bessel potential}
    \| \langle \nabla_\alpha \rangle^{-s} f \|_{L^r} \leq \| f \|_{L^r}.
\end{equation}
Consequently, $W_\alpha^{s,r}$ is complete.
\end{lemma}
Linear dispersive behavior of \cref{mixedNLS} is reflected in both fixed-time dispersive estimates and time-averaged Strichartz estimates to which the rest of this section is devoted. 
\begin{definition}\label{functionspace}
Define the pair $(q,r) \in (2,\infty]\times [2,\infty)$ to be \textit{admissible} if $\frac{1}{q}+\frac{1}{r} = \frac{1}{2}$. For $s\in \mathbb{R}$, $I \subseteq \mathbb{R}$ and $\alpha = \frac{2\alpha_2}{\alpha_1}$, let
\begin{equation*}
\begin{split}
    \| f \|_{S^s_{q,r}(I)} &= \| D_1^{-\beta_1(\frac{1}{2}-\frac{1}{r})} D_2^{-\beta_2(\frac{1}{2}-\frac{1}{r})} f \|_{L^q_{t\in I} W_\alpha^{s,r}}\\
    \| f \|_{\Tilde{S}^s_{q,r}(I)} &= (\sum_{N}\|  P_N |\nabla_\alpha|^{-(1+\frac{\alpha_1}{\alpha_2}-\alpha_1)(\frac{1}{2}-\frac{1}{r})} f \|^2_{L^q_{t\in I} W_\alpha^{s,r}})^{1/2},  
\end{split}
\end{equation*}
for smooth $f$ with compact support. Define the Strichartz space $S^s_{q,r}(I),\Tilde{S}^s_{q,r}(I)$ as the closure of functions under the norms above, respectively.
\end{definition}
\begin{proposition}[Fixed Time Estimate]\label{disp est2}
For $\theta \in [0,1]$,
\begin{equation*}
\begin{split}
    \| D_1^{-\beta_1 \theta}D_2^{-\beta_2 \theta} U(t)f \|_{L^{\frac{2}{1-\theta}}} &\lesssim |t|^{-\theta} \| f\|_{L^{\frac{2}{1+\theta}}}.
\end{split}
\end{equation*}
\end{proposition}
\begin{proof}[Proof of \cref{disp est2}]
Due to scaling, it suffices to prove the estimate at $t=1$. Use the complex interpolation method on the analytic family of linear operators $T_z = D_1^{-\beta_1 z}D_2^{-\beta_2 z} U(1)$ in the strip $\{z \in \mathbb{C}: \Re(z) \in [0,1]\}$ to obtain the desired estimate.\footnote{Showing the hypotheses of Stein's Interpolation Theorem has to be dealt with some care due to the pole of the symbol of $D_1^{-\beta_1 z}D_2^{-\beta_2 z}$ at the origin, but all operations are justified due to $\beta < 1$.} Let $\mu \in \mathbb{R}$. For $z = i\mu$, $T_z$ is bounded on $L^2$ by the Plancherel Theorem. For $z=1+i\mu$, $T_z:L^1\rightarrow L^\infty$ is bounded by \cref{disp est3}. 
\end{proof}
Averaging the estimate of \cref{disp est2} over time, the Strichartz estimates are obtained by the standard $TT^*$ argument.
\begin{proposition}[Strichartz Estimate]\label{strichartz4}
Let $(q,r),(\Tilde{q},\Tilde{r})$ be admissible. Then,
\begin{equation}\label{Strichartz1}
\begin{split}
    \| D_1^{-\beta_1(\frac{1}{2}-\frac{1}{r})} D_2^{-\beta_2(\frac{1}{2}-\frac{1}{r})}U(t)f\|_{L^q_t L^r(\mathbb{R}\times \mathbb{R}^2)} &\lesssim \| f\|_{L^2}\\
    \| \int U(t-\tau)F(\tau)d\tau \|_{L^q_t L^r} &\lesssim \| D_1^{\beta_1(1-\frac{1}{r}-\frac{1}{\Tilde{r}})} D_2^{\beta_2(1-\frac{1}{r}-\frac{1}{\Tilde{r}})}F \|_{L^{\Tilde{q}^\prime}_t L^{\Tilde{r}^\prime}}.
\end{split}
\end{equation}
\begin{equation}\label{Strichartz2}
    \begin{split}
        \| \int_{0}^{t} U(t-\tau)F(\tau)d\tau \|_{L^q_t L^r} \lesssim \| D_1^{\beta_1(1-\frac{1}{r}-\frac{1}{\Tilde{r}})} D_2^{\beta_2(1-\frac{1}{r}-\frac{1}{\Tilde{r}})}F \|_{L^{\Tilde{q}^\prime}_t L^{\Tilde{r}^\prime}}.
    \end{split}
\end{equation}
\end{proposition}
\begin{proof}
We use the $TT^*$ method on $T: L^2\rightarrow L^q_tL^r(\mathbb{R}\times \mathbb{R}^2)$ where $Tf := D_1^{-\frac{\beta_1\theta}{2}}D_2^{-\frac{\beta_2\theta}{2}} U(t)f$. Let $\theta \in [0,1)$ such that $(q,r)  = \Big(\frac{2}{\theta},\frac{2}{1-\theta} \Big)$. Given a spacetime function $F$, our task reduces to showing $\| TT^* F \|_{L^q_t L^r} \lesssim \| F \|_{L^{q^\prime}_t L^{r^\prime}}$. By the triangle inequality, \cref{disp est2}, and the Hardy-Littlewood-Sobolev inequality, 
\begin{equation*}
    \begin{split}
        \left|\left| TT^* F \right|\right|_{L^q_t L^r} &= \left|\left| \int D_1^{-\beta_1 \theta} D_2^{-\beta_2 \theta} U(t-\tau) F(\tau)d\tau\right|\right|_{L^q_t L^r}\leq \left|\left| \int \|D_1^{-\beta_1 \theta} D_2^{-\beta_2 \theta}U(t-\tau)F(\tau) \|_{L^r}d\tau\right|\right|_{L^q_t}\\
        &\lesssim \left|\left| \int |t-\tau|^{-\theta} \|F(\tau) \|_{L^{r^\prime}}d\tau \right|\right|_{L^q_t} \lesssim \| F \|_{L^{q^\prime}_t L^{r^\prime}}.
    \end{split}
\end{equation*}
This shows \cref{Strichartz1}. By the standard argument by the Christ-Kiselev Lemma \cite{christ2001maximal}, \cref{Strichartz2} follows.
\end{proof}
\begin{remark}
For $\theta = 1$, we have $(q,r) = (2,\infty)$, which constitutes the endpoint case in the application of the Hardy-Littlewood-Sobolev inequality.  The classical Strichartz estimate fails for this $(q,r)$ in $d=2$; see \cite{montgomery1998time}. It is of interest to investigate whether the analogue of the negative result for the classical Schr\"odinger evolution extends to our case.
\end{remark}
\begin{corollary}\label{disp est4}
For $s\in \mathbb{R}$, $\alpha = \frac{2\alpha_2}{\alpha_1}$, and admissible $(q,r)$,
\begin{equation}
\begin{split}
    \| U(t)f \|_{S^s_{q,r}(I)} ,\:\| U(t)f \|_{\Tilde{S}^s_{q,r}(I)}&\lesssim \| f \|_{H^s_\alpha}\\
    \| \int_0^t U(t-\tau)F(\tau)d\tau \|_{S^s_{q,r}(I)},\: \| \int_0^t U(t-\tau)F(\tau)d\tau \|_{\Tilde{S}^s_{q,r}(I)}&\lesssim \| F \|_{L^1_{t\in I}H^s_\alpha}
\end{split}
\end{equation}
\end{corollary}
\begin{proof}
From \cref{strichartz4}, replace $f,F$ by $\langle \nabla_\alpha \rangle^s f,\langle \nabla_\alpha \rangle^s F$ to obtain the estimates in $S^s_{q,r}(I)$. Further replacing $f,F$ by $\langle \nabla_\alpha \rangle^s P_N f, \langle \nabla_\alpha \rangle^s P_N F$ and summing over $N \in 2^\mathbb{Z}$, the estimates in $\Tilde{S}^s_{q,r}(I)$ are obtained, which we show in detail for the readers' convenience.  
It is shown that
\begin{equation}\label{disp est5}
    \sup_{x,y}\left|\iint e^{-it(|\xi|^{\alpha_1}+|\eta|^{\alpha_2}) + i(x\xi+y\eta)}\phi\Big(\frac{\sqrt{\xi^2+|\eta|^\alpha}}{N}\Big) d\xi d\eta\right| := \sup_{x,y}|I|\lesssim_{\alpha_1,\alpha_2} |t|^{-1} N^{1+\frac{\alpha_1}{\alpha_2}-\alpha_1}.
\end{equation}
Indeed by scaling, it suffices to show \cref{disp est5} for $N=1$. Let 
\begin{equation*}
\Psi(\xi,y,t) = \int e^{-it|\eta|^{\alpha_2}+iy\eta}\phi\Big(\sqrt{\xi^2+|\eta|^\alpha}\Big)d\eta = \int_{-c}^c e^{-it|\eta|^{\alpha_2}+iy\eta}\phi\Big(\sqrt{\xi^2+|\eta|^\alpha}\Big)d\eta,
\end{equation*}
where $c=4^{\frac{1}{\alpha}}$. By the support condition of $\phi$, $\Psi=0$ for $|\xi|>2$. The Van der Corput Lemma \cite[Eq 6, p.334]{stein1993harmonic} is applied to estimate the integral in \cref{disp est5}. More precisely, the phase function of interest is $\xi \mapsto -|\xi|^{\alpha_1} + \frac{x \xi}{t}$. Noting that $|\partial_{\xi}^2 (-|\xi|^{\alpha_1} + \frac{x \xi}{t})| \gtrsim 1$ on $\xi \in [-c,c]$,
\begin{equation}\label{eta}
\begin{split}
    \sup_x |I| &\lesssim |t|^{-\frac{1}{2}}\int_{-c}^c |\partial_\xi \Psi(\xi,y,t)| d\xi\\
    &=|t|^{-\frac{1}{2}} \int_{-c}^c |\xi| \left|\int_{-c}^c e^{-it|\eta|^{\alpha_2}+iy\eta}\frac{\phi^\prime (\sqrt{\xi^2+|\eta|^\alpha})}{\sqrt{\xi^2+|\eta|^\alpha}}d\eta\right| d\xi
\end{split}
\end{equation}
A similar argument applies to the $\eta$-integral in \cref{eta}. By the support condition of $\phi$, we have $\xi^2 + |\eta|^\alpha \simeq 1$, and by direct computation,
\begin{equation*}
    \left| \partial_\eta \frac{\phi^\prime (\sqrt{\xi^2+|\eta|^\alpha})}{\sqrt{\xi^2+|\eta|^\alpha}}\right| \lesssim |\eta|^{\alpha-1},
\end{equation*}
uniformly in $\xi$. Another application of the Van der Corput Lemma on the $\eta$-integral then yields the claim.

By the Young's inequality, the frequency-localized dispersive estimate
\begin{equation*}
    \| U(t) P_N f \|_{L^\infty} \lesssim |t|^{-1} N^{1+\frac{\alpha_1}{\alpha_2}-\alpha_1} \| P_N f \|_{L^1}
\end{equation*}
holds. Define $\Tilde{U}(t) = P_N U(N^{1+\frac{\alpha_1}{\alpha_2}-\alpha_1} t)P_{\sim N}$. Then,
\begin{equation}\label{refinement}
    \| \Tilde{U}(t)\Tilde{U}^\ast (\tau) f \|_{L^\infty} \lesssim |t-\tau|^{-1} \| f \|_{L^1},\: \| \Tilde{U}(t) f \|_{L^2} \lesssim \| f \|_{L^2}.
\end{equation}
By applying \cite[Theorem 1.2]{keel1998endpoint} to \cref{refinement}, we obtain the frequency-localized Strichartz estimates corresponding to $\Tilde{U}(t)$. Changing variables $t \mapsto N^{-({1+\frac{\alpha_1}{\alpha_2}-\alpha_1})}t$, we obtain
\begin{equation}\label{kt}
\begin{split}
    \| P_N U(t) f \|_{L^q_t L^r} &\lesssim N^{(1+\frac{\alpha_1}{\alpha_2}-\alpha_1) (\frac{1}{2}-\frac{1}{r})} \| P_N f \|_{L^2}\\
    \| \int_0^t P_N U(t-\tau)F(\tau)d\tau \|_{L^q_t L^r} &\lesssim N^{(1+\frac{\alpha_1}{\alpha_2}-\alpha_1)(1-\frac{1}{r}-\frac{1}{\Tilde{r}})}\| P_N F \|_{L^{\Tilde{q}^\prime}_{t}L^{\Tilde{r}^\prime}}.
\end{split}
\end{equation}
Observing that $\sum\limits_{N \in 2^\mathbb{Z}}N^{2(1+\frac{\alpha_1}{\alpha_2}-\alpha_1) (\frac{1}{2}-\frac{1}{r})} \| P_N f \|_{L^2}^2 \simeq \| |\nabla_\alpha|^{(1+\frac{\alpha_1}{\alpha_2}-\alpha_1)(\frac{1}{2}-\frac{1}{r})}f \|_{L^2}^2$ by the Plancherel Theorem, the refined estimates in $\Tilde{S}^s_{q,r}(I)$ follow by squaring \cref{kt} and summing in $N \in 2^\mathbb{Z}$.
\end{proof}
\section{Nonlinear Estimates.}\label{wp}
This section is devoted to the estimates needed to control the nonlinearity to close the contraction mapping argument. The goal is to obtain a solution $u \in C_T H^s_\alpha :=C_t H^s_\alpha ([0,T]\times \mathbb{R}^2)$ that satisfies the integral representation of \cref{mixedNLS}, which motivates the construction of
\begin{equation}\label{functionspace2}
X^s\coloneqq L^\infty_{t\in I}H_\alpha^{s}\cap S^{s}_{q,r}(I),\: \Tilde{X}^s\coloneqq \Tilde{S}^{s}_{\infty,2}(I)\cap \Tilde{S}^{s}_{q,r}(I).   
\end{equation}
In the scaling-subcritical regime, the $\| \cdot \|_{L^{p-1}_{t\in I}L^\infty}$ norm is controlled by the Sobolev embedding.
\begin{lemma}\label{sobolev embedding2}
For $2\leq r<q\leq \infty$, $s>(1+\frac{\alpha_1}{\alpha_2})(\frac{1}{2} - \frac{1}{q}) - \alpha_1(\frac{1}{2}-\frac{1}{r})$, and $\alpha = \frac{2\alpha_2}{\alpha_1}$,
\begin{equation}\label{sobemb}
    \| u \|_{L^q}  \lesssim \| D_1^{-\beta_1(\frac{1}{2}-\frac{1}{r})} D_2^{-\beta_2(\frac{1}{2}-\frac{1}{r})}u \|_{W^{s,r}_\alpha}.
\end{equation}
\end{lemma}
\begin{proof}
We first show 
\begin{equation}\label{sobolevembedding}
\| u \|_{L^q} \lesssim \| D_1^{-\beta_1(\frac{1}{2}-\frac{1}{r})} D_2^{-\beta_2(\frac{1}{2}-\frac{1}{r})}u \|_{L^r} + \| D_1^{-\beta_1(\frac{1}{2}-\frac{1}{r})} D_2^{-\beta_2(\frac{1}{2}-\frac{1}{r})}u \|_{\dot{W}_\alpha^{s,r}}=:X.    
\end{equation} 
By the triangle inequality followed by \cref{Bernstein},
\begin{equation*}
    \| u \|_{L^q} \leq \sum_{N,N_1,N_2} \| P_N P^{(1)}_{N_1}P^{(2)}_{N_2} u \|_{L_q} \lesssim \sum_{N,N_1,N_2} (N_1 N_2)^{\frac{1}{r}-\frac{1}{q}}\| P_N P^{(1)}_{N_1}P^{(2)}_{N_2} u \|_{L^r}. 
\end{equation*}
By \cref{Bernstein3,Bernstein5},
\begin{equation}
    \begin{split}
    \| P_N P^{(1)}_{N_1}P^{(2)}_{N_2} u \|_{L^r} &= \| D_1^{\beta_1(\frac{1}{2}-\frac{1}{r})} D_2^{\beta_2(\frac{1}{2}-\frac{1}{r})} P_N P^{(1)}_{N_1}P^{(2)}_{N_2} D_1^{-\beta_1(\frac{1}{2}-\frac{1}{r})} D_2^{-\beta_2(\frac{1}{2}-\frac{1}{r})}u\|_{L^r}\\
    &\lesssim N_1^{\beta_1(\frac{1}{2}-\frac{1}{r})} N_2^{\beta_2(\frac{1}{2}-\frac{1}{r})} X,
    \end{split}
\end{equation}
and similarly,
\begin{equation}
    \begin{split}
    \| P_N P^{(1)}_{N_1}P^{(2)}_{N_2} u \|_{L^r} &\simeq N^{-s} N_1^{\beta_1(\frac{1}{2}-\frac{1}{r})} N_2^{\beta_2(\frac{1}{2}-\frac{1}{r})} \|P_N P^{(1)}_{N_1}P^{(2)}_{N_2} D_1^{-\beta_1(\frac{1}{2}-\frac{1}{r})} D_2^{-\beta_2(\frac{1}{2}-\frac{1}{r})} |\nabla_\alpha|^{s}u \|_{L^r}\\ &\lesssim N^{-s} N_1^{\beta_1(\frac{1}{2}-\frac{1}{r})} N_2^{\beta_2(\frac{1}{2}-\frac{1}{r})} X.
    \end{split}
\end{equation}
Combining the two inequalities above, we obtain
\begin{equation*}
    \|P_N P^{(1)}_{N_1}P^{(2)}_{N_2} u \|_{L^q} \lesssim (N_1 N_2)^{\frac{1}{r}-\frac{1}{q}} N_1^{\beta_1(\frac{1}{2}-\frac{1}{r})}N_2^{\beta_2(\frac{1}{2}-\frac{1}{r})} \min(1,N^{-s})X.
\end{equation*}
Summing in $N_1,N_2$, 
\begin{equation*}
    \sum_{N_1 \lesssim N,\:N_2 \lesssim N^{\frac{2}{\alpha}}} (N_1 N_2)^{\frac{1}{r}-\frac{1}{q}} N_1^{\beta_1(\frac{1}{2}-\frac{1}{r})}N_2^{\beta_2(\frac{1}{2}-\frac{1}{r})} \lesssim  N ^{(1+\frac{\alpha_1}{\alpha_2})(\frac{1}{2} - \frac{1}{q}) - \alpha_1(\frac{1}{2}-\frac{1}{r})}.
\end{equation*}
Summing in $N$,
\begin{equation*}
    \sum_N N ^{(1+\frac{\alpha_1}{\alpha_2})(\frac{1}{2} - \frac{1}{q}) - \alpha_1(\frac{1}{2}-\frac{1}{r})} \min(1,N^{-s}) \lesssim 1,
\end{equation*}
by $(1+\frac{\alpha_1}{\alpha_2})(\frac{1}{2} - \frac{1}{q}) - \alpha_1(\frac{1}{2}-\frac{1}{r})>0$ by direct computation and the condition on $s$. This shows \cref{sobolevembedding}.

By \cref{BesselPotential}, $\| D_1^{-\beta_1(\frac{1}{2}-\frac{1}{r})} D_2^{-\beta_2(\frac{1}{2}-\frac{1}{r})}u \|_{L^r} \leq \| D_1^{-\beta_1(\frac{1}{2}-\frac{1}{r})} D_2^{-\beta_2(\frac{1}{2}-\frac{1}{r})}u \|_{W_\alpha^{s,r}}$. To show
that the inhomogeneous derivative controls the homogeneous derivative, it suffices to prove
\begin{equation}\label{derivative}
    \| \langle \nabla_\alpha \rangle^{-s} |\nabla_\alpha|^s f \|_{L^r} \lesssim \| f \|_{L^r},
\end{equation}
for any $s \geq 0$ and $r\in [1,\infty]$. Since $\langle \nabla_\alpha \rangle^{-s} |\nabla_\alpha|^s$ is the identity operator for $s=0$, assume $s>0$. Arguing as \cite[V.3 Lemma 2]{stein2016singular}, there exists $\mu_s$, a finite complex Borel measure on $\mathbb{R}^2$, such that $\langle \nabla_\alpha \rangle^{-s} |\nabla_\alpha|^s f = \mu_s \ast f$, and this shows the desired boundedness.
\end{proof}
\begin{remark}\label{sobolev algebra}
If one formally substitutes $\alpha_1=\alpha_2=2$, then one recovers the classical Sobolev embedding $W^{s,r}\hookrightarrow L^q$. Moreover if $s>\frac{1}{2}+\frac{1}{\alpha}$, it can be shown as in the proof of \cref{sobolev embedding2} that the continuous embedding $H^s_\alpha \hookrightarrow L^\infty$ holds and that $H^s_\alpha$ is an algebra. 
\end{remark}
Since the estimate we will need corresponds to $q=\infty$, we state it as a corollary.
\begin{corollary}\label{sobolev embedding}
For $2 \leq r<\infty$, $s>\frac{1}{2} + \frac{\alpha_1}{2\alpha_2} - \alpha_1 (\frac{1}{2}-\frac{1}{r})$, and $\alpha = \frac{2\alpha_2}{\alpha_1}$,
\begin{equation}\label{sobemb2}
    \| u \|_{L^\infty} \lesssim \| D_1^{-\beta_1(\frac{1}{2}-\frac{1}{r})}D_2^{-\beta_2(\frac{1}{2}-\frac{1}{r})}u \|_{W_\alpha^{s,r}}.
\end{equation}
Moreover, if $s>s_c$ and $p \geq 3$, then there exists $r\in (\frac{1}{\frac{s-s_c}{\alpha_1}+\frac{p-3}{2(p-1)}},\infty)$ such that $s>\frac{1}{2} + \frac{\alpha_1}{2\alpha_2} - \alpha_1 (\frac{1}{2}-\frac{1}{r})$.
\end{corollary}
\begin{proof}
By direct computation, $s>\frac{1}{2} + \frac{\alpha_1}{2\alpha_2} - \alpha_1 (\frac{1}{2}-\frac{1}{r})$ holds if and only if
\begin{equation*}
    s-s_c > \alpha_1(\frac{1}{p-1}-\frac{1}{2}+\frac{1}{r}).
\end{equation*}
The conclusion follows from a straightforward computation.
\end{proof}
\begin{remark}
The Littlewood-Paley decomposition can be viewed as a vector-valued operator $Tf = (P_N f)_{N\in2^\mathbb{Z}}$ defined on $\mathcal{S}$, the class of Schwartz functions. For $\alpha=2$ and $r \in (1,\infty)$, $T:L^r(\mathbb{R}^2)\rightarrow L^r(\mathbb{R}^2;l^2(2^\mathbb{Z}))$ is bounded, a consequence of $T$ being a Calder\'on-Zygmund operator (CZO). However for $\alpha \neq 2$, $T$ is not a CZO. To see this, consider the integral representation
\begin{equation*}
\begin{split}
Tf(x,y) = \iint K(x,y;x^\prime,y^\prime)f(x^\prime,y^\prime)dx^\prime dy^\prime,\:K(x,y;x^\prime,y^\prime) = (\check{\phi}_N(x-x^\prime,y-y^\prime))_{N\in 2^\mathbb{Z}},
\end{split}
\end{equation*}
for $f \in \mathcal{S}$. If $T$ defines a CZO, then $K$ obeys a decay estimate of the form
\begin{equation*}
    \| K(x,y;x^\prime,y^\prime)\|_{l^2_N} \lesssim (|x-x^\prime|^2+|y-y^\prime|^2)^{-1}.
\end{equation*}
In particular, for $x=x^\prime=y^\prime=0$ and $y \in \mathbb{R}$,
\begin{equation*}
    (\sum_{N\in 2^\mathbb{Z}}|\check{\phi}_N(0,y)|^2)^{1/2} \lesssim |y|^{-2},
\end{equation*}
and therefore $    |\check{\phi}_N(0,y)|  = N^{1+\frac{2}{\alpha}}|\check{\phi}_1(0,N^{\frac{2}{\alpha}}y)|\lesssim |y|^{-2}$, or equivalently, $|\check{\phi}_1(0,y)| \lesssim N^{-(1-\frac{2}{\alpha})}|y|^{-2}$, for every $N \in 2^\mathbb{Z}$. By taking $N \to 0+$ for $\alpha \in (0,2)$, we obtain $\check{\phi}_1(0,y)=0$ for all $y \in \mathbb{R}\setminus \{0\}$, and thus $\check{\phi}_1(0,0) = \iint \phi(\sqrt{\xi^2+|\eta|^\alpha})d\xi d\eta=0$ by continuity, a contradiction. A similar argument holds for $\alpha \in (2,\infty)$ by taking $N \to \infty$.
\end{remark}
In the scaling-critical regime, the $\| \cdot \|_{L^{p-1}_{t\in I}L^\infty}^{p-1}$ norm is controlled by the Besov-refined Strichartz norms. The restriction $p>3$ in \cref{thm2} is due to the loss of derivatives in the Strichartz estimates in \cref{disp est4}. In fact, the restriction $p \geq 3$ seems unavoidable even in the sub-critical regime since the only Strichartz estimates on the inhomogeneous term without any derivative loss occurs when $(\Tilde{q},\Tilde{r}) = (\infty,2)$. The proof of the following lemma is adapted from that of \cite[Lemma 3.5]{1534-0392_2015_6_2265}.
\begin{lemma}\label{besov}
For $p>3$, let $(q,r)$ be admissible and $2<q<p-1$. Then,
\begin{equation*}
    \| u \|_{L^{p-1}_{t\in I}L^\infty}^{p-1} \lesssim \| u \|_{\Tilde{S}^{s_c}_{q,r}(I)}^{q} \| u \|_{\Tilde{S}^{s_c}_{\infty,2}(I)}^{p-1-q}. 
\end{equation*}
\end{lemma}
In applying the Strichartz estimates on the nonlinear term, we need mixed-derivative analogues of some well-known nonlinear estimates on the power nonlinearity. Modifying the proofs of \cite[Lemma A.8, Proposition A.9]{tao2006nonlinear}, we have
\begin{lemma}\label{nonlin.est.}
For $k \in \mathbb{N}$, let $F \in C^{k-1,1}_{loc}(\mathbb{C};\mathbb{C})$ and $F(0)=0$ where $\mathbb{C}$ is identified with $\mathbb{R}^2$. For any $s \in [0,k]$, if $u \in H^s_\alpha \cap L^\infty$, then
\begin{equation*}
    \| F(u) \|_{H^s_\alpha} \lesssim_{s,k,\alpha, \| u\|_{L^\infty}} \| u \|_{H^s_\alpha}.
\end{equation*}
If $F(u) = |u|^{p-1}u$ for $p>1$, then for any $s \in [0,\lfloor p \rfloor]$ for $p$ not an odd integer or $s \in [0,\infty)$ for $p$ an odd integer,
\begin{equation}\label{Schauder}
    \| F(u) \|_{H^s_\alpha} \lesssim_{s,p,\alpha} \| u \|_{L^\infty}^{p-1} \| u\|_{H^s_\alpha}.
\end{equation}
If $F(u) = |u|^{p-1}$ for $p > 2$, then for any $s \in [0,\lfloor p \rfloor-1]$ for $p$ not an odd integer or $s \in [0,\infty)$ for $p$ an odd integer,
\begin{equation}\label{Schauder2}
    \| F(u) \|_{H^s_\alpha} \lesssim_{s,p,\alpha} \| u \|_{L^\infty}^{p-2} \| u\|_{H^s_\alpha}.
\end{equation}
\end{lemma}
The proof of \cref{nonlin.est.} is shown in the Appendix. Since \cref{leibniz} can be proved similarly, we omit its proof.
\begin{lemma}\label{leibniz}
For $s \geq 0$,
\begin{equation*}
    \| fg \|_{H^s_\alpha} \lesssim_{s,\alpha} \| f \|_{H^s_\alpha} \| g \|_{L^\infty} + \| f \|_{L^\infty} \| g \|_{H^s_\alpha},
\end{equation*}
for all $f,g \in H^s_\alpha \cap L^\infty$.
\end{lemma}
As such, the analogue of the Leibniz rule in Sobolev spaces holds in anisotropic Sobolev spaces as well. In relation to the Kato-Ponce inequality, it is of independent interest whether the Leibniz rule holds in the $L^p$-based anisotropic Sobolev spaces when $p \neq 2$.  
\section{Well-posedness.}\label{gwp}
This section is devoted to the proofs of \cref{thm1,thm2}. Moreover global well-posedness for $s_c<\frac{\alpha_1}{2}$ is discussed.
\begin{lemma}[Persistence of Regularity]\label{persistence}
For $0\leq s_1 \leq s_2$ and $R>0$, $B_R = \{u: \|u \|_{X^{s_2}}\leq R\}$ is complete under $\| \cdot \|_{X^{s_1}}$.
\end{lemma}
\begin{proof}
See \cite[Lemma 3.2]{1534-0392_2015_6_2265} and \cite[Theorem 1.2.5]{cazenave2003semilinear}. By \cref{BesselPotential}, $W_\alpha^{s,r}$ is complete for $s \geq 0$. Moreover, it is reflexive for $r \in (1,\infty)$.
\end{proof}
\begin{proof}[Proof of \cref{thm1}]
Let $I = [0,T]$ for $T>0$ to be determined. For $s\in\mathbb{R}$ that satisfies the hypothesis of \cref{thm1}, there exists an admissible $(q,r)$ such that $q>\max(p-1,2)$ where $r$ is as in \cref{sobolev embedding}. We wish to show that
\begin{equation}\label{contraction}
    \Gamma u = U(t)u_0 -i\mu \int_0^t U(t-t^\prime)|u(t^\prime)|^{p-1} u(t^\prime)dt^\prime
\end{equation}
defines a contraction on $X^s$. From \cref{Strichartz1}, the linear estimate follows immediately:
\begin{equation*}
    \| U(t) u_0 \|_{L^\infty_{t\in I}H_\alpha^{s}}, \| U(t)u_0\|_{S^{s}_{q,r}(I)} \lesssim \| u_0 \|_{H_\alpha^{s}}.
\end{equation*}
For the nonlinear estimate, use \cref{Strichartz2,nonlin.est.} to obtain
\begin{equation}\label{nonlinear estimate}
    \|\int_0^t U(t-t^\prime)|u(t^\prime)|^{p-1} u(t^\prime)dt^\prime \|_{X^s} \lesssim \| |u|^{p-1}u\|_{L^1_{t\in I} H_\alpha^{s}} \lesssim \| u \|_{L^{p-1}_{t\in I}L^\infty}^{p-1}\| u \|_{L^\infty_{t\in I}H_\alpha^{s}}.
\end{equation}
By the H\"older's inequality in $t$ and \cref{sobolev embedding}, the RHS of \cref{nonlinear estimate} is estimated above by
\begin{equation*}
\begin{split}
    &\lesssim T^{0+} \| u \|_{L^{q}_{t\in I}L^\infty}^{p-1} \| u \|_{L^\infty_{t\in I}H_\alpha^{s}}\lesssim T^{0+} \| u \|_{S^{s}_{q,r}}^{p-1}\| u \|_{L^\infty_{t\in I}H_\alpha^{s}} \lesssim T^{0+} \| u \|_{X^s}^p.
\end{split}
\end{equation*}
Hence by taking $T>0$ sufficiently small such that $T^{0+}\| u_0 \|_{H_\alpha^{s}}^{p-1}\ll 1$, we have $\Gamma: \Omega \rightarrow \Omega$ where $\Omega = \{ u \in X^s: \| u \|_{X^s} \leq 2 \| u_0 \|_{H_\alpha^{s}}\}$. Showing that $\Gamma$ defines a contraction on $\Omega$ follows estimates similar to the ones we have done previously. It should be noted that the contraction is proved on $\Omega$ equipped with the $\| \cdot \|_{X^0}$ norm, which is complete by \cref{persistence}.
\begin{equation*}
\begin{split}
\| \Gamma u - \Gamma v \|_{X^0} &\lesssim \| |u|^{p-1}u - |v|^{p-1}v\|_{L^1_{t\in I}L^2}\lesssim \| (|u|^{p-1}+|v|^{p-1})|u-v|\|_{L^1_{t\in I}L^2}\\
&\lesssim  (\| u \|_{L^{p-1}_{t\in I}L^\infty}^{p-1}+ \| v \|_{L^{p-1}_{t\in I}L^\infty}^{p-1})\| u-v\|_{L^\infty_{t\in I}L^2}\\
&\lesssim T^{0+} (\|u \|_{S^{s}_{q,r}(I)}^{p-1} + \|v \|_{S^{s}_{q,r}(I)}^{p-1}) \| u -v \|_{L^\infty_{t\in I}L^2}\lesssim T^{0+} (\|u \|_{X^s}^{p-1} + \|v \|_{X^s}^{p-1}) \| u -v \|_{X^0}.
\end{split}
\end{equation*}
By taking $T\ll 1$ depending on $\| u_0 \|_{H^{s}_\alpha}$, $\Gamma$ is a contraction in $(\Omega, \| \cdot \|_{X^0})$, and therefore there exists a unique fixed point in $\Omega$.

To show that the solution map is continuous, let $u_0,u_{0,n} \in H^s_\alpha$ for $n \in \mathbb{N}$ with $u_{0,n} \xrightarrow[n\to \infty]{H^s_\alpha}u_0$ and let $u,u_n$ be the solutions corresponding to $u_0,u_{0,n}$, respectively. Since the time of existence depends on the $H^s_\alpha$-norm of data, there exists $T>0$ such that $u,u_n \in X^s$ for sufficiently large $n$. Then by \cref{disp est4} and the triangle inequality,
\begin{equation*}
\begin{split}
    \| u-u_n \|_{X^s} &\lesssim \| u_0 - u_{0,n} \|_{H^s_\alpha} + \| |u|^{p-1}u - |u_n|^{p-1}u_n \|_{L^1_t H^s_\alpha}.
\end{split}
\end{equation*}
By the Mean Value Theorem and \cref{leibniz},
\begin{equation*}
\begin{split}
    \| |u|^{p-1}u - |u_n|^{p-1}u_n \|_{L^1_t H^s_\alpha} &\lesssim
    (\| u \|_{L^{p-1}_t L^\infty}^{p-1}+\| u_n \|_{L^{p-1}_t L^\infty}^{p-1}) \| u-u_n \|_{L^\infty_t H^s_\alpha}\\
    &+(\| u \|_{L^{p-1}_t L^\infty}^{p-2}+\| u_n \|_{L^{p-1}_t L^\infty}^{p-2})(\| u \|_{L^\infty_t H^s_\alpha}+\| u_n \|_{L^\infty_t H^s_\alpha})\| u-u_n \|_{L^{p-1}_t L^\infty}\\
    &\lesssim T^{0+} \Big(\| u \|_{S^s_{q,r}}^{p-1}+\| u_n \|_{S^s_{q,r}}^{p-1}  + (\| u \|_{S^s_{q,r}}^{p-2}+\| u_n \|_{S^s_{q,r}}^{p-2})(\| u \|_{X^s} + \| u_n \|_{X^s})\Big)\| u-u_n \|_{X^s}.
\end{split}
\end{equation*}
Hence,
\begin{equation*}
    \| u - u_n \|_{X^s} \leq C_1 \| u_0 - u_{0,n} \|_{H^s_\alpha} + C_2 T^{0+} \| u-u_n \|_{X^s},
\end{equation*}
where $C_2$ depends on $\| u_0 \|_{H^s_\alpha}$. By shrinking $T>0$ if necessary, we obtain
\begin{equation*}
    \| u-u_n \|_{X^s} \leq C \| u_0 - u_{0,n} \|_{H^s_\alpha},
\end{equation*}
and the claim follows.
\end{proof}
\begin{remark}
Note that in the proof, it is assumed that $s \in (s_c,\lfloor p \rfloor]$, for $p$ not an odd integer, to obtain a unique solution in $X^s$ whereas the assumption strengthens to $s \in (s_c,\lfloor p \rfloor - 1]$ to prove the continuous dependence on initial data. 
\end{remark}
\begin{proof}[Proof of \cref{thm2}]
By the contraction mapping theorem, the existence of a unique solution is obtained in $\Tilde{X}^{s_c}$ where $(q,r)$ is as \cref{besov}. To provide the main ideas of the proof, let $\Gamma$ be as \cref{contraction}. By a priori estimates
\begin{equation*}
\begin{split}
    \| \Gamma u \|_{\Tilde{S}^{s_c}_{\infty,2}(I)} &\lesssim \| u_0 \|_{H^{s_c}_\alpha} + \| u \|^q_{\Tilde{S}^{s_c}_{q,r}(I)}\| u \|^{p-q}_{\Tilde{S}^{s_c}_{\infty,2}(I)}\\
    \| \Gamma u \|_{\Tilde{S}^{s_c}_{q,r}(I)} &\lesssim \| U(t) u_0 \|_{\Tilde{S}^{s_c}_{q,r}(I)} + \| u \|^q_{\Tilde{S}^{s_c}_{q,r}(I)}\| u \|^{p-q}_{\Tilde{S}^{s_c}_{\infty,2}(I)},   
\end{split}
\end{equation*}
$T>0$ is chosen such that $\| U(t)u_0 \|_{\Tilde{S}^{s_c}_{q,r}(I)} \leq c$ where $c \simeq \| u_0 \|_{H^{s_c}_\alpha}^{-\epsilon}$ and $\epsilon=\epsilon(p,q,\alpha_1,\alpha_2)>0$. Then, it can be shown that $\Gamma$ defines a contraction on
\begin{equation}\label{contraction2}
M = \{u \in \Tilde{X}^{s_c}: \| u \|_{\Tilde{S}^{s_c}_{\infty,2}(I)} \lesssim \| u_0 \|_{H^{s_c}_\alpha},\: \| u \|_{\Tilde{S}^{s_c}_{q,r}(I)}\lesssim c\},    
\end{equation}
where the implicit constants depend on $p,q$ and the constants from the Strichartz estimates, which depend on $\alpha_1,\alpha_2$. The rest of the argument is standard, and thus we omit it.

To show small data scattering, observe that there exists $\delta=\delta(p,q,\alpha_1,\alpha_2)>0$ such that whenever $\| u_0 \|_{H^{s_c}_\alpha} < \delta$, it follows that $\| U(t)u_0 \|_{\Tilde{S}^{s_c}_{q,r}([0,T])} \leq \| U(t)u_0 \|_{\Tilde{S}^{s_c}_{q,r}(\mathbb{R})} \lesssim \| u_0 \|_{H^{s_c}_\alpha}\leq c$. Taking $T$ to be arbitrarily large, $u$ extends globally in time. Arguing as \cite[Theorem 1.3]{1534-0392_2015_6_2265}, it can be shown that $\lim \limits_{t\to\pm \infty}U(-t)u(t)=:u_{\pm}(\alpha)$ is convergent in $H^{s_c}_\alpha$ from which \cref{scattering} follows. 
\end{proof}

In the rest of this section, it is shown that the finite energy solution can be extended globally in time under some hypotheses. Let $\alpha = \frac{2\alpha_2}{\alpha_1}$. In \cref{conservation}, the derivative terms are controlled if $u_0 \in H_\alpha^{\frac{\alpha_1}{2}}$. The second term, or the nonlinear energy, can be controlled by the Gagliardo-Nirenberg inequality for mixed derivatives. By \cref{gagliardo}, the nonlinear energy is finite for $u_0 \in H_\alpha^{\frac{\alpha_1}{2}}$ if $p<\frac{\alpha_1^{-1}+\alpha_2^{-1}+1}{\alpha_1^{-1}+\alpha_2^{-1}-1}$, or equivalently if $s_c<\frac{\alpha_1}{2}$.
\begin{lemma}[Gagliardo-Nirenberg Inequality]\label{GN inequality}
Let $s>0,\: 1<p<q\leq \infty$ and $0<\theta<1$ satisfy $s\theta = (1+\frac{2}{\alpha})(\frac{1}{p}-\frac{1}{q})$. Then,
\begin{equation}\label{gagliardo}
    \| u \|_{L^q} \lesssim_{s,p,q,\alpha} \| u \|_{\dot{W}_\alpha^{s,p}(\mathbb{R}^2)}^\theta \| u \|_{L^p}^{1-\theta}.
\end{equation}
\end{lemma}
\begin{remark}
The defocusing nonlinearity ($\mu=1$) immediately yields global existence by the conservation of mass and $\| |\nabla_\alpha|^{\frac{\alpha_1}{2}}u(t)\|_{L^2} \lesssim \sqrt{E[u_0]}$.
\end{remark}
In the focusing case ($\mu=-1$), however, the global existence of solution is not expected for every data, as it is the case for the classical NLS. For instance, the focusing cubic NLS on $\mathbb{R}^2$ is $L^2$-critical and it is known (see \cite{weinstein1982nonlinear}) that this equation is globally well-posed for $u_0 \in H^1$ with $\| u_0 \|_{L^2}< \| \phi \|_{L^2}$ where $\phi \in \mathcal{S}$ is the ground state solution to the (appropriately scaled) PDE
\begin{equation}\label{soliton}
    \Delta \psi + \psi^3 = \psi.
\end{equation}
We use the Gagliardo-Nirenberg inequality for mixed derivatives to bootstrap the local theory to the global theory in $H_\alpha^{\frac{\alpha_1}{2}}$.
\begin{corollary}\label{global existence}
Consider the focusing \cref{mixedNLS} and let $p,\alpha = \frac{2\alpha_2}{\alpha_1}$ satisfy $s_c < \frac{\alpha_1}{2}$. If $(p-1)(\frac{1}{\alpha_1}+\frac{1}{\alpha_2})<2$, then the solution exists globally in time for any $u_0 \in H^{\frac{\alpha_1}{2}}_\alpha$. If $(p-1)(\frac{1}{\alpha_1}+\frac{1}{\alpha_2})=2$, then there exists $C = C(p,\alpha_1,\alpha_2)>0$ such that whenever $\| u_0 \|_{L^2} < C$, the $H^{\frac{\alpha_1}{2}}_\alpha$-solution exists for all $t$. If $(p-1)(\frac{1}{\alpha_1}+\frac{1}{\alpha_2})>2$, there exists $E_1=E_1(p,\alpha_1,\alpha_2)>0$ and $\epsilon = \epsilon(p,\alpha_1,\alpha_2)$ such that whenever $\| u_0 \|_{H^{\frac{\alpha_1}{2}}_\alpha}< \epsilon$ and $0<E[u_0]\leq E_1$, the $H^{\frac{\alpha_1}{2}}_\alpha$-solution exists for all $t$.
\end{corollary}
\begin{proof}
By the mass and energy conservation,
\begin{equation}\label{derivative growth}
\begin{split}
    \| |\nabla_\alpha|^{\frac{\alpha_1}{2}} u \|_{L^2}^2 &= 2 E[u_0] + \frac{2}{p+1}\| u \|_{L^{p+1}}^{p+1}\\
    &\leq C( E[u_0] +  \| |\nabla_\alpha|^{\frac{\alpha_1}{2}} u \|_{L^2}^{(p-1)(\frac{1}{\alpha_1}+\frac{1}{\alpha_2})}\| u_0 \|_{L^2}^{\delta}),
\end{split}
\end{equation}
where $C>0$ is by \cref{GN inequality} and $\delta>0$. If $(p-1)(\frac{1}{\alpha_1}+\frac{1}{\alpha_2})<2$, then $\| |\nabla_\alpha|^{\frac{\alpha_1}{2}}u(t)\|_{L^2}$ is bounded in time, and therefore the local theory for any $u_0 \in H^{\frac{\alpha_1}{2}}_\alpha$ can be iterated with uniform time steps. If $(p-1)(\frac{1}{\alpha_1}+\frac{1}{\alpha_2})=2$ and $\| u_0 \|_{L^2}$ is sufficiently small depending on $p,\alpha_1,\alpha_2$, then $\| |\nabla_\alpha|^{\frac{\alpha_1}{2}}u(t)\|_{L^2}$ is bounded in time, and hence the global existence of solution. Lastly assume $(p-1)(\frac{1}{\alpha_1}+\frac{1}{\alpha_2})>2$. Observe that \cref{derivative growth} is in the form
\begin{equation}\label{derivative growth2}
x \leq C( E[u_0]+ x^{\gamma} \| u_0 \|_{L^2}^{\delta}),
\end{equation}
where $x = x(t) = \| |\nabla_\alpha|^{\frac{\alpha_1}{2}} u(t) \|_{L^2}^2$ and $\gamma>1\: ,\delta>0$. Let $f(x,y) = C(y+x^\gamma \| u_0 \|_{L^2}^{\delta})$. By direct computation, $f$, as a function of $x$, coincides tangentially with the identity at
\begin{equation*}
(x_0,y_0) = \Big((C\gamma)^{-\frac{1}{\gamma-1}}\| u_0 \|_{L^2}^{-\frac{\delta}{\gamma-1}},(C\gamma)^{-\frac{\gamma}{\gamma-1}}(\gamma-1)\| u_0 \|_{L^2}^{-\frac{\delta}{\gamma-1}}\Big).    
\end{equation*}
Let $\epsilon = (C\gamma)^{-\frac{1}{(\gamma-1)(2+\frac{\delta}{\gamma-1})}}$ and suppose $\| u_0 \|_{L^2}, \| |\nabla_\alpha|^{\frac{\alpha_1}{2}} u_0\|_{L^2}<\epsilon$. The smallness assumption of $u_0$ implies 
\begin{equation*}
    x(0)=\| |\nabla_\alpha|^{\frac{\alpha_1}{2}} u_0\|_{L^2}^2 < \epsilon^2<x_0,
\end{equation*}
and moreover
\begin{equation*}
    y_0 = (C\gamma)^{-1}(\gamma-1)x_0 > (C\gamma)^{-1}(\gamma-1) \epsilon^2 =: E_1(p,\alpha).
\end{equation*}
Hence the derivative $\| |\nabla_\alpha|^{\frac{\alpha_1}{2}} u(t)\|_{L^2}$ is bounded in $t$.
\end{proof}
\section{Regularity in the Dispersion Parameter.}\label{alpha regularity}
This section is devoted to the proofs of \cref{limit local,limit global}.
\begin{proof}[Proof of \cref{limit local}]
Denote $U_\alpha(t) = e^{-it(|\xi|^{\alpha_1}+|\eta|^{\alpha_2})}$ and similarly for $U_{\alpha^\prime}(t)$. Let $\sigma = \sigma(\xi,\eta;s,\alpha) = (1+\xi^2+|\eta|^\alpha)^{\frac{s}{2}}$. Since $U_\alpha(t),U_{\alpha^\prime}(t)$ are isometries on $H^s_\alpha$ and by the algebra property of $H^s_\alpha$ (see \cref{sobolev algebra}), \cref{leibniz}, and the Mean Value Theorem, we have
\begin{equation}\label{algebra}
    \| |u|^{p-1}u - |v|^{p-1}v \|_{H^s_\alpha} \lesssim_{p,s,\alpha} (\| u \|_{L^\infty}^{p-1}+\| v \|_{L^\infty}^{p-1})\| u-v \|_{H^s_\alpha}.
\end{equation}
It thus follows from the standard semi-group theory that \cref{continuity} is locally well-posed in $H^s_\alpha$. 

Let $\epsilon>0$ and consider the difference equation:
\begin{equation*}
    u^\alpha(t) - u^{\alpha^\prime}(t) = U_{\alpha}(t)u_{0,\alpha} - U_{\alpha^\prime}(t)u_{0,\alpha^\prime} - i\mu \int_0^t \Big(U_\alpha(t-t^\prime)(|u^\alpha(t^\prime)|^{p-1}u^\alpha(t^\prime))-U_{\alpha^\prime}(t-t^\prime)(|u^{\alpha^\prime}(t^\prime)|^{p-1}u^{\alpha^\prime}(t^\prime)) \Big)dt^\prime.
\end{equation*}
The linear contribution can be estimated above within $\epsilon$ for $|\alpha_1-\alpha_1^\prime|+|\alpha_2 - \alpha_2^\prime|$ sufficiently small. By the triangle inequality,
\begin{equation*}
\begin{split}
 \| U_{\alpha}(t)u_{0,\alpha} - U_{\alpha^\prime}(t)u_{0,\alpha^\prime} \|_{H^s_\alpha} &\leq \|(U_\alpha(t)-U_{\alpha^\prime}(t))u_{0,\alpha} \|_{H^s_\alpha} + \|U_{\alpha^\prime}(t)(u_{0,\alpha}-u_{0,\alpha^\prime}) \|_{H^s_\alpha}\\
 &= \|(U_\alpha(t)-U_{\alpha^\prime}(t))u_{0,\alpha} \|_{H^s_\alpha} + \|u_{0,\alpha}-u_{0,\alpha^\prime} \|_{H^s_\alpha}.
\end{split}
\end{equation*}
Note that
\begin{equation*}
\|(U_\alpha(t)-U_{\alpha^\prime}(t))u_{0,\alpha} \|_{H^s_\alpha}\leq T \| w(\xi,\eta)\sigma \widehat{u_{0,\alpha}}\|_{L^2(K)} + 2 \| \sigma \widehat{u_{0,\alpha}}\|_{L^2(K^c)}  
\end{equation*}
where $w(\xi,\eta)=\left||\xi|^{\alpha_1}-|\xi|^{\alpha_1^\prime}\right|+\left||\eta|^{\alpha_2}-|\eta|^{\alpha_2^\prime}\right|$ and $K \subseteq \mathbb{R}^2$ is taken sufficiently big depending on $\epsilon$ and $u_{0,\alpha}$.

The nonlinear contribution is estimated above by
\begin{equation*}
\begin{split}
    \int_0^t \|(U_\alpha(t-t^\prime)-U_{\alpha^\prime}(t-t^\prime))(|u^\alpha(t^\prime)|^{p-1}u^\alpha(t^\prime)) \|_{H^s_\alpha} dt^\prime + \int_0^t \| |u^\alpha(t^\prime)|^{p-1}u^\alpha(t^\prime)-|u^{\alpha^\prime}(t^\prime)|^{p-1}u^{\alpha^\prime}(t^\prime)\|_{H^s_\alpha} dt^\prime\\
    :=I+II.    
\end{split}
\end{equation*}
Observing that $|u^\alpha(t^\prime)|^{p-1} u^\alpha(t^\prime) \in C([0,T];H^s_\alpha)$ by the algebra property of $H^s_\alpha$, if $U_{\alpha^\prime}(t)f(t) \xrightarrow[]{} U_\alpha(t)f(t)$, as $(\alpha_1^\prime,\alpha_2^\prime)\to (\alpha_1,\alpha_2)$, in $H^s_\alpha$ uniformly in $[0,T]$ for every $f \in C([0,T];H^s_\alpha)$, then $I \leq \epsilon T$ for $|\alpha_1 - \alpha_1^\prime|+|\alpha_2 - \alpha_2^\prime|$ sufficiently small depending on $T$ and $f=|u^\alpha|^{p-1}u^\alpha$. To show this hypothesis, we have 
\begin{equation*}
\begin{split}
    \| (U_\alpha(t)-U_{\alpha^\prime}(t))f(t)\|_{H^s_\alpha}^2 &\leq 2 \iint \Big(|e^{-it|\xi|^{\alpha_1}}-e^{-it|\xi|^{\alpha_1^\prime}}|^2+ |e^{-it|\eta|^{\alpha_2}}-e^{-it|\eta|^{\alpha_2^\prime}}|^2\Big) \sigma^2 |\widehat{f(t)}|^2d\xi d\eta\\
    &=4 \iint \Big((1-\cos{t(|\xi|^{\alpha_1} - |\xi|^{\alpha_1^\prime})})+(1-\cos{t(|\eta|^{\alpha_2} - |\eta|^{\alpha_2^\prime})})\Big)\sigma^2 |\widehat{f(t)}|^2 d\xi d\eta.
\end{split}
\end{equation*}
It suffices to show
\begin{equation*}
    \lim\limits_{\alpha_1^\prime \rightarrow \alpha_1}\iint (1-\cos{t(|\xi|^{\alpha_1} - |\xi|^{\alpha_1^\prime})})\sigma^2 |\widehat{f(t)}|^2 d\xi d\eta =: \lim\limits_{\alpha_1^\prime \rightarrow \alpha_1} III = 0,
\end{equation*}
uniformly in $t \in [0,T]$. For $n \in \mathbb{N}$, define $\xi_n = \xi_n(\alpha_1^\prime)$ to be the first positive $\xi$ such that
\begin{equation*}
1-\cos{T(|\xi|^{\alpha_1} - |\xi|^{\alpha_1^\prime})} = \frac{1}{n}.    
\end{equation*}
Define $S_n = [-\xi_n,\xi_n]^c \times \mathbb{R}_\eta \subseteq \mathbb{R}^2$. Then,
\begin{equation*}
\begin{split}
    III = \iint_{S_n^c} + \iint_{S_n}\leq \frac{\| f \|_{L^\infty_{t^\prime\in [0,T]}H^s_\alpha}}{n} + 2 \iint_{S_n} \sigma^2 |\widehat{f(t)}|^2 d\xi d\eta.    
\end{split}
\end{equation*}
Fix $N \in \mathbb{N}$ such that the first term above with $n=N$ is bounded above by $\epsilon$. To show that the second term is also bounded above by $\epsilon$ uniformly in $t^\prime \in [0,T]$, define $F_{\alpha_1^\prime}(t) = \iint_{S_N} \sigma^2 |\widehat{f(t)}|^2 d\xi d\eta$. Since
\begin{equation*}
\begin{split}
    |F_{\alpha_1^\prime}(t_1)-F_{\alpha_1^\prime}(t_2)| &\leq \iint_{S_N} \sigma^2 |\widehat{f(t_1)}-\widehat{f(t_2)}|\cdot (|\widehat{f(t_1)}|+|\widehat{f(t_2)}|) d\xi d\eta\\
    &\leq 2 \| f \|_{L^\infty_{t^\prime\in [0,T]}H^s_\alpha} \| f(t_1)-f(t_2) \|_{H^s_\alpha},
\end{split}
\end{equation*}
and $\lim \limits_{\alpha_1^\prime \to \alpha_1}F_{\alpha_1^\prime}(t)=0$ pointwise by the Dominated Convergence Theorem, the claim follows by the Arzel\`a-Ascoli Theorem.

As for the term $II$, the argument leading to \cref{algebra} yields
\begin{equation*}
\begin{split}
    \| |u^\alpha|^{p-1}u^\alpha - |u^{\alpha^\prime}|^{p-1}u^{\alpha^\prime} \|_{H^s_\alpha} \lesssim (\| u^\alpha \|_{H^s_\alpha}^{p-1} + \| u^{\alpha^\prime} \|_{H^s_\alpha}^{p-1}) \| u^\alpha-u^{\alpha^\prime} \|_{H^s_\alpha}.
\end{split}
\end{equation*}
By the H\"older's inequality,
\begin{equation*}
\begin{split}
    II &\lesssim (\| u^\alpha \|_{L^\infty_{t^\prime\in [0,T]}H^s_\alpha}^{p-1}+\| u^{\alpha^\prime} \|_{L^\infty_{t^\prime\in [0,T]}H^s_\alpha}^{p-1})\int_0^t \| u^\alpha(t^\prime) - u^{\alpha^\prime}(t^\prime) \|_{H^s_\alpha} dt^\prime\\
    &\lesssim \| u_{0,\alpha} \|_{H^s_\alpha}^{p-1} exp\Big(C(\alpha)(p-1)\| u^\alpha \|_{L^{p-1}_{t^\prime} L^\infty([0,T]\times \mathbb{R}^2)}^{p-1}\Big)\int_0^t \| u^\alpha(t^\prime) - u^{\alpha^\prime}(t^\prime) \|_{H^s_\alpha} dt^\prime,    
\end{split}
\end{equation*}
for all $\alpha_1^\prime,\alpha_2^\prime$ with $|\alpha_1 - \alpha_1^\prime|+|\alpha_2 - \alpha_2^\prime|$ sufficiently small where the last inequality follows from applying the Gronwall's inequality to \cref{uniform bound} with $C(\alpha)>0$ from \cref{nonlin.est.}.

Therefore for every $\epsilon>0$, there exists $\delta>0$ that depends on $\epsilon, s,p,\alpha_1,\alpha_2,T,u_{0,\alpha}$ such that whenever $|\alpha_1 - \alpha_1^\prime|+|\alpha_2 - \alpha_2^\prime|<\delta$, it follows that
\begin{equation*}
    \| u^\alpha(t)-u^{\alpha^\prime}(t) \|_{H^s_\alpha} \lesssim \epsilon(1+T) + \| u_{0,\alpha} \|_{H^s_\alpha}^{p-1} \int_0^t \| u^\alpha(t^\prime) - u^{\alpha^\prime}(t^\prime) \|_{H^s_\alpha} dt^\prime,
\end{equation*}
and the desired claim follows from the Gronwall's inequality.
\end{proof}
While energy estimates in Sobolev algebras on a compact time interval is sufficient to keep $u^\alpha(t)-u^{\alpha^\prime}(t)$ small for $|\alpha_1-\alpha_1^\prime|+|\alpha_2-\alpha_2^\prime|$ small, this method based on the Gronwall's inequality fails to capture the oscillatory nature of dispersive phenomena, and thus it is unlikely that the approach in \cref{limit local} would extend to $T=\infty$ uniformly in time.

In the following discussion, let $\alpha_1 = 2$ for simplicity, and so $\alpha = \frac{2\alpha_2}{\alpha_1} = \alpha_2$. We utilize \cref{thm2} to illustrate the pivotal role played by phase decoherence in the long-time dynamics of \cref{mixedNLS} for different dispersive parameters. More precisely, there exists a non-trivial datum $\phi$ such that $u^{\alpha^\prime}[\phi](t)$ does not converge to $u^\alpha[\phi](t)$ uniformly in $t\in [0,\infty)$ as $\alpha^\prime \to \alpha$ where the notations $u^\alpha,u^{\alpha^\prime}$ are as \cref{continuity}. For $\alpha=2,\: p=3$, note that the smallness condition (formally) given by \cref{thm2} is in $L^2$, which is consistent with \cite[Theorem A]{weinstein1982nonlinear}.
\begin{proposition}\label{limit global}
Assume the hypotheses of \cref{thm2}. Then there exist $R=R(\alpha)>0$ and $\phi \in \mathcal{S}\setminus \{0\}$ such that whenever $\alpha^\prime\in (0,2]\setminus \{1\}$ satisfies $|\alpha^\prime - \alpha |\leq R$, $u^{\alpha^\prime}[\phi]$ exists globally in time and scatters to free solutions in $L^2$. Furthermore,
\begin{equation}\label{discontinuity}
\lim\limits_{\alpha^\prime\to\alpha} \|u^\alpha-u^{\alpha^\prime} \|_{L^\infty_t L^2(\mathbb{R}\times \mathbb{R}^2)} >0.
\end{equation}
\end{proposition}
\begin{proof}
For $\alpha \in (0,2)\setminus \{1\}$, the implicit constants obtained in \cref{Bernstein,Bernstein5,Bernstein3,nonlin.est.,disp est3,disp est4}, call them $C(\alpha)$, are locally stable in $\alpha$ in the sense that there exists $R=R(\alpha)>0$ sufficiently small such that if $|\alpha- \alpha^\prime|\leq R$, then $C(\alpha^\prime) \in [\frac{C(\alpha)}{2},2C(\alpha)]$; this $R>0$ cannot be taken arbitrarily large, since certain estimates (for example, see \cref{critpt,critpt2}) blow up as $\alpha^\prime$ tends to zero or one. For $\alpha=2$, the same conclusion holds observing that $\varlimsup\limits_{\alpha^\prime \rightarrow 2-} C(\alpha^\prime)<\infty$ and that the aforementioned lemmas with $\alpha=2$ hold. The local stability of implicit constants implies that if $\alpha \in (0,2]\setminus \{1\}$, then there exists $\delta_\alpha>0$ such that whenever $\sup\limits_{\alpha^\prime: |\alpha-\alpha^\prime|\leq R}\| \phi \|_{H^{s_c(\alpha^\prime)}_{\alpha^\prime}} \leq \delta_\alpha$, solutions to \cref{mixedNLS}, with the dispersion parameter $(\alpha_1,\alpha_2) = (2,\alpha^\prime)$ and the datum $\phi$, exist globally in time and scatter to free solutions in $H^{s_c(\alpha^\prime)}_{\alpha^\prime}$, and therefore, in $L^2$. Let $0<|\alpha - \alpha^\prime|\leq R$ and denote $u_+(\alpha),u_+(\alpha^\prime) \in L^2$ by the corresponding asymptotic states. Then,
\begin{equation*}
    \| u^\alpha(t) - U_\alpha(t)u_+(\alpha) \|_{L^2},\: \| u^{\alpha^\prime}(t) - U_{\alpha^\prime}(t)u_+(\alpha^\prime) \|_{L^2} \xrightarrow[t\rightarrow \infty]{} 0,
\end{equation*}
where $\| u_+(\alpha) \|_{L^2} = \| u_+(\alpha^\prime) \|_{L^2} = \| \phi \|_{L^2}$ by the mass conservation. By the triangle inequality,
\begin{equation*}
    \| u^\alpha(t) - u^{\alpha^\prime}(t) \|_{L^2} \geq \| U_\alpha(t)u_+(\alpha) - U_{\alpha^\prime}(t)u_+(\alpha^\prime) \|_{L^2}-( \|u^\alpha(t) - U_\alpha(t)u_+(\alpha) \|_{L^2}+ \|u^{\alpha^\prime}(t)-U_{\alpha^\prime}(t)u_+(\alpha^\prime) \|_{L^2}).
\end{equation*}
From the identity
\begin{equation*}
    \| U_\alpha(t)u_+(\alpha) - U_{\alpha^\prime}(t)u_+(\alpha^\prime) \|_{L^2}^2 = 2\| \phi \|_{L^2}^2 - 2 Re \langle U_\alpha(t)u_+(\alpha),U_{\alpha^\prime}(t)u_+(\alpha^\prime)\rangle,
\end{equation*}
\cref{discontinuity} follows if 
\begin{equation*}
\lim\limits_{t\rightarrow\infty}\langle U_\alpha(t)u_+(\alpha),U_{\alpha^\prime}(t)u_+(\alpha^\prime)\rangle=\lim\limits_{t\rightarrow\infty}\iint e^{-it(|\eta|^\alpha-|\eta|^{\alpha^\prime})} \widehat{u_+(\alpha)}\cdot\overline{\widehat{u_+(\alpha^\prime)}}d\xi d\eta = 0,    
\end{equation*}
which in turn follows from
\begin{equation*}
    \lim\limits_{t\rightarrow\infty} I(t):=\lim\limits_{t\rightarrow\infty} \int e^{-it(|\eta|^\alpha-|\eta|^{\alpha^\prime})} f(\eta)d\eta=0,
\end{equation*}
for all $f \in C^\infty_c(\mathbb{R})$ by density. Let $\phi(\eta) = |\eta|^\alpha - |\eta|^{\alpha^\prime}$. Since
\begin{equation}\label{phase}
    \phi^\prime(\eta) = sgn(\eta)(\alpha |\eta|^{\alpha-1} - \alpha^\prime |\eta|^{\alpha^\prime-1}),\: \phi^{\prime\prime}(\eta) = \alpha (\alpha-1) |\eta|^{\alpha-2} - \alpha^\prime (\alpha^\prime-1) |\eta|^{\alpha^\prime-2},
\end{equation}
the critical points are $0,\pm \eta_c$ where $\eta_c := (\frac{\alpha^\prime}{\alpha})^{\frac{1}{\alpha-\alpha^\prime}}$; note that $\phi^{\prime\prime}(0)$ is undefined whereas $\phi^{\prime\prime}(\eta_c) \neq 0$ for every $\alpha^\prime \neq \alpha$. From \cref{phase}, there exists $\nu = \nu(\alpha,\alpha^\prime) \in (0,\frac{\eta_c}{2})$ such that $\sup\limits_{\eta \in [0,\nu]} |\phi^{\prime\prime}(\eta)|\geq 1$. Define smooth bump functions taking values in $[0,1]$ as follows:
\begin{equation*}
\begin{split}
\psi_1(\eta) &=
\begin{cases*}
      1 &, $\eta \in [0,\frac{\nu}{2}]$\\
      0 &, $\eta \in (\nu,\infty)$,
    \end{cases*}\\
\psi_2(\eta) &=
\begin{cases*}
      1 &, $\eta \in [\nu,2\eta_c]$\\
      0 &, $\eta \in [0,\frac{\nu}{2})\cup (4\eta_c,\infty)$\\
      1-\psi_1(\eta) &, $\eta \in [\frac{\nu}{2},\nu]$.
    \end{cases*} 
\end{split}
\end{equation*}
By the Van der Corput Lemma\footnote{Though $\phi$ is not smooth at the origin, the lemma holds for our particular phase function.},
\begin{equation*}
    \left|\int_0^\nu e^{-it\phi(\eta)}\psi_1(\eta)f(\eta)d\eta\right| \lesssim_{\alpha,\alpha^{\prime}} |t|^{-\frac{1}{2}}.
\end{equation*}
By the method of stationary phase,
\begin{equation*}
    \left|\int_{\frac{\nu}{2}}^{4\eta_c} e^{-it\phi(\eta)}\psi_2(\eta)f(\eta)d\eta\right|  \lesssim_{\alpha,\alpha^{\prime}} |t|^{-\frac{1}{2}}.
\end{equation*}
Lastly by the method of non-stationary phase,
\begin{equation*}
    \left|\int_{0}^{\infty} e^{-it\phi(\eta)}(1-\psi_1(\eta)-\psi_2(\eta))f(\eta)d\eta\right|  \lesssim_{\alpha,\alpha^{\prime},k} |t|^{-k},\:\text{for all}\: k\in \mathbb{N}.
\end{equation*}
This shows $\int_0^\infty e^{-it\phi(\eta)}f(\eta)d\eta\xrightarrow[t\rightarrow\infty]{}0$ and the integral on $(-\infty,0]$ can be shown similarly by the change of variable $\eta \mapsto -\eta$.
\end{proof}
\begin{remark}\label{singular limit}
In \cref{mixedNLS}, certain regimes of dispersive parameters are of interest. Our study contains $\alpha_1 = \alpha_2$ as a special case. Since the non-dispersive solutions ($\alpha_1 = \alpha_2 =1$) do not exhibit small-data scattering (see \cite{krieger2013nondispersive}) whereas dispersive solutions for $\alpha_1 = \alpha_2 \neq 1$ do (see \cite{1534-0392_2015_6_2265}), it is of interest to ask the same question when $\alpha_1 =1,\:\alpha_2 \neq 1$. It is also of interest whether the ODE limit $(\alpha_1,\alpha_2) \rightarrow (0,0)$ reveals any interesting features of \cref{mixedNLS} for small $\alpha_i$. Lastly we remark that although the implicit constants of the Strichartz estimates possibly blow up as $\alpha_i \rightarrow 0$ or $1$ (for example, see \cref{critpt,critpt2}), the measure of non-locality (or the loss of derivatives) measured by $\beta_i = 1-\frac{\alpha_i}{2}$ is smooth in $\alpha_i$.
\end{remark}
\section{Conclusions.}
To establish the local well-posedness theory of mNLSE, the functional framework that respects the spatially-anisotropic scaling symmetry was developed. In doing so, the standard Littlewood-Paley theory based on smooth projections was extended to non-smooth projections. It is of interest to study the long-time and blow-up dynamics of mNLSE corresponding to large data. It is also of interest how our work ties back to applications to nonlinear optics and photonics. 

\section{Acknowledgements.}
Both authors work is supported by the U.S. National
Science Foundation under the grants DMS-1909559 and RTG, DMS-1840260.
\appendix
\section{Appendix}\label{appendix}
\begin{proof}[Proof of \cref{Bernstein3}]
Define a smooth bump function $\zeta \in [0,1]$ that is identically one on $supp(\phi_1)$ and compactly supported in the $\delta$-neighborhood of $supp(\phi_1)$ for small $\delta>0$ so that $supp(\zeta) \subseteq B(0,\epsilon)^c$ for some $\epsilon>0$. Define $\zeta_N(\xi,\eta) = \zeta\Big(\frac{\xi}{N},\frac{\eta}{N^{2/\alpha}}\Big)$ and denote $f = P_N u$. Since $\widehat{f} = \zeta_N \widehat{f}$, we have $f = N^{1+\frac{2}{\alpha}}\check{\zeta}(Nx,N^{\frac{2}{\alpha}}y)\ast f$. Hence by the Young's inequality and chain rule,
\begin{equation*}
    \| |\nabla_\alpha|^{s} f\|_{L^r} = N^{1+\frac{2}{\alpha}+s}\| (|\nabla_\alpha|^s \check{\zeta})(Nx,N^{\frac{2}{\alpha}}y) \ast f \|_{L^r} \leq N^s\| |\nabla_\alpha|^s \check{\zeta}\|_{L^1}\| f \|_{L^r}.  
\end{equation*}
Since $(\xi^2 + |\eta|^\alpha)^{\frac{s}{2}}\zeta$ has a compact support on which it is smooth in $\xi$ and sufficiently regular in $\eta$, $|\nabla_\alpha|^s \check{\zeta}$ obeys an estimate of the form \cref{decay1,decay2}, and thus $\| |\nabla_\alpha|^s \check{\zeta}\|_{L^1} \lesssim_s 1$. Conversely,
\begin{equation*}
    \| P_N u \|_{L^r} = \| |\nabla_\alpha|^{-s}|\nabla_\alpha|^{s} P_N u \|_{L^r} \lesssim N^{-s} \| |\nabla_\alpha|^s P_N u \|_{L^r}.
\end{equation*}

Similarly, define $\Tilde{\zeta}$ to be a smooth bump function identically one on $\{|\xi_i| \in [\frac{1}{2},2]\}\subseteq \mathbb{R}$ and supported in $\{|\xi_i| \in [\frac{1}{4},4]\}$. Let $\Tilde{\zeta}_{N_i}(\xi_1,\xi_2) = \Tilde{\zeta}(\frac{\xi_i}{N_i})$. By arguing as above, we obtain the second estimate.
\end{proof}
\begin{proof}[Proof of \cref{BesselPotential}]
The claim is trivial for $s=0$, and so assume $s>0$. From the definition of Gamma function,
\begin{equation*}
    \lambda^{-\frac{s}{2}} = \Gamma(\frac{s}{2})^{-1} \int_0^\infty e^{-t\lambda} t^{\frac{s}{2}-1}dt,
\end{equation*}
for $\lambda>0$. For $\lambda =1+\xi^2+|\eta|^\alpha$,
\begin{equation*}
    (1+\xi^2+|\eta|^\alpha)^{-\frac{s}{2}} = \Gamma(\frac{s}{2})^{-1}\int_0^\infty e^{-t} t^{\frac{s}{2}-1}e^{-t\xi^2}e^{-t|\eta|^\alpha}dt.
\end{equation*}
Since it is known that the inverse Fourier transform of $e^{-|\eta|^\alpha}$ for $0<\alpha\leq 2$ is non-negative, we conclude $G_s(x,y) := \mathcal{F}^{-1}[(1+\xi^2+|\eta|^\alpha)^{-\frac{s}{2}}](x,y) \geq 0$, and therefore
\begin{equation*}
\| G_s \|_{L^1} = \iint G_s(x,y)dxdy = \widehat{G_s}(0,0) = 1,
\end{equation*}
and the desired estimate follows from the Young's inequality.

To show completeness, the claim is immediate for $r=2$ by the Plancherel Theorem. For $r \neq 2$, if $\{f_n\}_{n=1}^\infty$ is a sequence such that $\|f_n - f_m\|_{W_\alpha^{s,r}}\xrightarrow[n,m\to \infty]{}0$, then there exists $F \in L^r$ such that $\langle \nabla_\alpha \rangle^s f_n \xrightarrow[n\to \infty]{L^r} F$. Since $\langle \nabla_\alpha \rangle^{-s} F \in L^r\subseteq \mathcal{S}^\prime$ by \cref{bessel potential}, the Cauchy sequence converges to $\langle \nabla_\alpha \rangle^{-s} F$ in $W_\alpha^{s,r}$.
\end{proof}
It is known in \cite[Proposition 1]{cho2011remarks} that the frequency-localized analogue of \cref{disp est3} holds for $\alpha_1 = \alpha_2$. When $\alpha=2$, the Strichartz estimates (\cref{strichartz4,disp est4}) follows immediately since the embedding $B^0_{r,2} \hookrightarrow L^r$ is bounded for $r \geq 2$. When $\alpha<2$, however, it remains to show that such Besov refinement remains true. Instead we directly show the analogue of \cite[Proposition 1]{cho2011remarks} by the standard oscillatory phase argument where the integral is on the entire Fourier domain. More generally, it is of independent interest whether the dispersive estimates in \cite{cho2011remarks} could be extended to non-radial dispersion relations.
\begin{lemma}\label{disp est3}
Let $\mu \in \mathbb{R}$. There exists $C = C(\alpha_1,\alpha_2)>0$ such that
\begin{equation}\label{disp est1}
    \lVert D_1^{-\beta_1(1+i\mu)}D_2^{-\beta_2(1+i\mu)}U(t)f \rVert_{L^\infty(\mathbb{R}^2)} \leq C |t|^{-1}\| f \|_{L^1(\mathbb{R}^2)}.
\end{equation}
\end{lemma}
\begin{proof}
Without loss of generality, we prove the estimate at $t=1$. By showing
\begin{equation*}
    \iint e^{-i(|\xi|^{\alpha_1}+|\eta|^{\alpha_2}) + i(x\xi+y\eta)}|\xi|^{-\beta_1(1+i\mu)}|\eta|^{-\beta_2(1+i\mu)} d\xi d\eta \in L^\infty(\mathbb{R}^2),
\end{equation*}
the proof is immediate by the Young's inequality. The integrand is a product, and therefore it suffices to show
\begin{equation*}
    K(y)\coloneqq \int e^{-i |\eta|^\alpha + i y \eta}|\eta|^{-\beta(1+i\mu)}d\eta \in L^\infty (\mathbb{R}),
\end{equation*}
for $\alpha \in (0,2]\setminus\{1\}$ and $\beta = 1-\frac{\alpha}{2}$. Since $K$ is an even function, let $y \geq 0$.

\textbf{Case I: $\alpha>1$.}

Suppose $y \leq 1$. Since $\beta<1$, the integral on $[-R,R]$ is bounded uniformly for all $y \in \mathbb{R}$ by the triangle inequality where $R>0$ is some constant to be determined. Integrating by parts,
\begin{equation}\label{IBP}
    \int_R^\infty e^{-i\eta^\alpha+iy\eta}\eta^{-\beta(1+i\mu)}d\eta = -\int_R^\infty e^{-i\eta^\alpha+iy\eta} \partial_\eta \Big(\frac{\eta^{-\beta(1+i\mu)}}{i(-\alpha \eta^{\alpha-1}+y)}\Big)d\eta +O\Big(\frac{R^{-\beta}}{|-\alpha R^{\alpha-1}+y|}\Big).  
\end{equation}

There exists $R>0$ sufficiently big such that whenever $\eta \geq R$,
\begin{equation*}
    |-\alpha \eta^{\alpha-1} + y| \geq \alpha |\eta|^{\alpha-1} - |y| \geq |\eta|^{\alpha-1}, 
\end{equation*}
and therefore the boundary term is estimated above by $R^{-\beta - \alpha +1} = R^{-\frac{\alpha}{2}} = O(1)$.
As for the integral term, note that
\begin{equation*}
    \left| \partial_\eta \Big(\frac{\eta^{-\beta(1+i\mu)}}{-\alpha\eta^{\alpha-1}+y}\Big)\right| \lesssim \frac{\eta^{-(\beta+1)}}{|-\alpha\eta^{\alpha-1}+y|} + \frac{\eta^{-\beta+\alpha-2}}{|-\alpha\eta^{\alpha-1}+y|^2} \lesssim \eta^{-(\alpha+\beta)} = \eta^{-(1+\frac{\alpha}{2})},
\end{equation*}
and therefore the LHS of \cref{IBP} is bounded in $y$. Similarly the integral for $K(y)$ on $(-\infty,-R)$ can be shown to be uniformly bounded for $|y|\leq 1$ under the change of variable $\eta \mapsto -\eta$. 

Now suppose $y \geq 1$ and let $\Phi_y(\eta) \coloneqq -\frac{|\eta|^\alpha}{y}+\eta$ so that $K(y) = \int e^{iy \Phi_y(\eta)}|\eta|^{-\beta(1+i\mu)} d\eta$. Note that
\begin{equation*}
\begin{split}
    \Phi_y^\prime(\eta) = -\frac{\alpha |\eta|^{\alpha-1}sgn(\eta)}{y}+1;\:
    \Phi_y^{\prime\prime}(\eta)= -\frac{\alpha(\alpha-1)|\eta|^{\alpha-2}}{y}.
\end{split}
\end{equation*}
Define
\begin{equation}\label{critpt}
    \xi_0 = (\frac{y}{\alpha})^{\frac{1}{\alpha-1}},\: \xi_1 = 2^{-\frac{1}{\alpha-1}}\xi_0,\: \xi_2 = 2^{\frac{1}{\alpha-1}}\xi_0.
\end{equation}
By direct computation, one can verify
\begin{equation}\label{critpt2}
    \Phi_y^\prime(\xi_0) = 0,\: \Phi_y^{\prime\prime}(\xi_0) =-\alpha^{\frac{1}{\alpha-1}}(\alpha-1) y^{-\frac{1}{\alpha-1}}.
\end{equation}
In particular, $\xi_0$ is a non-degenerate critical point. We now estimate
\begin{equation*}
\begin{split}
I &= \int e^{iy\Phi_y(\eta)}|\eta|^{-\beta(1+i\mu)}\zeta(\eta)d\eta;\: II= \int_{\frac{\xi_1}{2}}^{\xi_1} e^{iy\Phi_y(\eta)}|\eta|^{-\beta(1+i\mu)}(1-\zeta(\eta))d\eta\\
III&= \int_{\xi_2}^{2\xi_2}e^{iy\Phi_y(\eta)}|\eta|^{-\beta(1+i\mu)}(1-\zeta(\eta))d\eta;\: IV=\int_{(-\infty,\frac{\xi_1}{2}]\cup [2\xi_2,\infty)} e^{iy\Phi_y(\eta)}|\eta|^{-\beta(1+i\mu)}d\eta,
\end{split}
\end{equation*}
where $\zeta \in C^\infty_c(\mathbb{R})$ is a smooth bump localized around $\xi_0$ defined by $\zeta(\eta) = \psi(\frac{\eta-\xi_0}{\xi_0})$ where a smooth bump function $0 \leq \psi \leq 1$ is given by
\begin{equation*}
\psi(\eta^\prime) =
\begin{cases*}
      1 &, $\eta^\prime \in [-1+2^{-\frac{1}{\alpha-1}},2^{\frac{1}{\alpha-1}}-1]$\\
      0 &, $\eta^\prime\leq -1 + 2^{-\frac{\alpha}{\alpha-1}}$ or $\eta^\prime \geq 2^{\frac{\alpha}{\alpha-1}}-1$.
    \end{cases*}
\end{equation*}
By inspection, $K(y) = I+II+III+IV$. Note that $\zeta$ depends on $y$ whereas $\psi$ depends only on $\alpha$. Changing variables and applying the method of stationary phase as \cite[Proposition 3, p.334]{stein1993harmonic} (in the support of $\psi$, $\Phi_y$ is smooth, and therefore the method of stationary phase can be applied), we obtain
\begin{equation*}
\begin{split}
    I &= \int e^{iy\Phi_y(\eta)}|\eta|^{-\beta(1+i\mu)}\psi(\frac{\eta-\xi_0}{\xi_0})d\eta= \xi_0 \int e^{iy\Phi_y(\xi_0 + \xi_0 \eta^\prime)}|\xi_0 + \xi_0 \eta^\prime|^{-\beta(1+i\mu)} \psi(\eta^\prime)d\eta^\prime\\
    &= Ce^{iy\Phi_y(\xi_0)}\psi(0)|\xi_0|^{-\beta(1+i\mu)} \Phi_y^{\prime\prime}(\xi_0)^{-\frac{1}{2}}y^{-\frac{1}{2}} + O_\psi(y^{-\frac{3}{2}})
\end{split}
\end{equation*}
as $y \to \infty$ for some $C>0$. By \cref{critpt,critpt2}, the dominant term is of order $y^0$. This shows $I = O(1)$. To show $II=O(1)$, we use the Van der Corput Lemma \cite[Eq 6, p.334]{stein1993harmonic}. Since $\Phi^\prime_y \geq \frac{1}{2}$ and is monotonic on $[\frac{\xi_1}{2},\xi_1]$ with $\zeta(\xi_1)=1$, we have
\begin{equation*}
\begin{split}
    |II| \lesssim y^{-1}\int_{\frac{\xi_1}{2}}^{\xi_1}\left|\partial_\eta (|\eta|^{-\beta(1+i\mu)}(1-\zeta(\eta)))\right|d\eta\lesssim y^{-1} \int_{\frac{\xi_1}{2}}^{\xi_1} \eta^{-(\beta+1)} + \eta^{-\beta}\xi_0^{-1}d\eta \lesssim y^{-1} \xi_0^{-\beta} \simeq y^{-(1+\frac{\beta}{\alpha-1})}.
\end{split}
\end{equation*}
Hence, $II = O(1)$ and $III=O(1)$ can be shown similarly.

Finally, $IV$ can be shown to be $O(1)$ in a similar way. An extra care is needed due to the singularity of $|\eta|^{-\beta(1+i\mu)}$ and the non-smoothness of $\Phi_y$ at the origin. This can be done by splitting the integral in regions $(-\infty,-1] \cup [-1,1] \cup [1,\frac{\xi_1}{2})\cup [2\xi_2,\infty)$. More precisely, one applies the Van der Corput Lemma on the first, third, and the fourth region, and the triangle inequality on $[-1,1]$ using $\beta<1$.

\textbf{Case II: $0<\alpha<1$.}

As before, the integral on $\eta \in [-\epsilon,\epsilon]$ is uniformly bounded in $y\in \mathbb{R}$ where $\epsilon>0$ is a constant to be determined. For the integral on $(-\infty,-\epsilon]$,  change variables $\eta \mapsto -\eta$ to obtain
\begin{equation}\label{IBP2}
    \int_\epsilon^\infty e^{-i\eta^\alpha-iy\eta}\eta^{-\beta(1+i\mu)}d\eta = -\int_\epsilon^\infty e^{-i\eta^\alpha-iy\eta} \partial_\eta \Big(\frac{\eta^{-\beta(1+i\mu)}}{i(-\alpha \eta^{\alpha-1}-y)}\Big)d\eta +O\Big(\frac{\epsilon^{-\beta}}{|-\alpha \epsilon^{\alpha-1}-y|}\Big).  
\end{equation}
Since $|-\alpha \eta^{\alpha-1}-y| \geq \alpha \eta^{\alpha-1}$ for all $\eta > 0$, one can reason as in the case $\alpha>1$ to show that the integral on $(-\infty,-\epsilon]$ is bounded in $y \in \mathbb{R}$. It remains to show that the integral on $\eta \in [\epsilon,\infty)$ is bounded in $y$.

First assume $y \leq 1$. For $\eta \in D:= (\epsilon,(\frac{2y}{\alpha})^{\frac{1}{\alpha-1}}) \cup ((\frac{2y}{3\alpha})^{\frac{1}{\alpha-1}},\infty)$, it follows from the triangle inequality that $|-\alpha \eta^{\alpha-1}-y| \geq \frac{\alpha}{2}\eta^{\alpha-1}$, and therefore the integral on $D$ is bounded in $y \leq 1$ by the integration by parts argument as in \cref{IBP2}.

On $\eta\in D_1:=((\frac{2y}{\alpha})^{\frac{1}{\alpha-1}},(\frac{2y}{3\alpha})^{\frac{1}{\alpha-1}})$, the integral is estimated by the method of stationary phase. For $\xi_0 := (\frac{y}{\alpha})^{\frac{1}{\alpha-1}}$, define $\zeta(\eta) = \psi(\frac{\eta-\xi_0}{\xi_0})$ where a smooth bump $0 \leq \psi \leq 1$ is given by
\begin{equation*}
\psi(\eta^\prime) =
\begin{cases*}
      1 &, $\eta^\prime \in [(\frac{4}{3})^{\frac{1}{\alpha-1}}-1,(\frac{4}{5})^{\frac{1}{\alpha-1}}-1]$\\
      0 &, $\eta^\prime\leq 2^{\frac{1}{\alpha-1}}-1$ or $\eta^\prime \geq (\frac{2}{3})^{\frac{1}{\alpha-1}}-1$.
    \end{cases*}
\end{equation*}
Then,
\begin{equation*}
\begin{split}
    \int_{D_1}e^{-i\eta^\alpha+iy\eta}\eta^{-\beta(1+i\mu)}d\eta &= \int_{D_1}e^{i\lambda \Phi_\lambda(\eta)}\eta^{-\beta(1+i\mu)}\zeta(\eta)d\eta+
    \int_{D_1}e^{-i\eta^\alpha+iy\eta}\eta^{-\beta(1+i\mu)}(1-\zeta(\eta))d\eta\\
    &=:A+B,
\end{split}
\end{equation*}
where
\begin{equation*}
\lambda := \frac{1}{y},\: \Phi_\lambda(\eta) := -\frac{\eta^\alpha}{\lambda} + \frac{\eta}{\lambda^2}.
\end{equation*}
By changing variables $\eta^\prime = \frac{\eta - \xi_0}{\xi_0}$,
\begin{equation*}
\begin{split}
    A &= \xi_0 \int_{2^{\frac{1}{\alpha-1}}-1}^{(\frac{2}{3})^{\frac{1}{\alpha-1}}-1} e^{i\lambda \Phi_\lambda(\xi_0 + \xi_0 \eta^\prime)}|\xi_0 + \xi_0 \eta^\prime|^{-\beta(1+i\mu)}\psi(\eta^\prime)d\eta^\prime\\
    &= C e^{i\lambda \Phi_\lambda(\xi_0)}|\xi_0|^{-\beta(1+i\mu)} \psi(0) \lambda^{\frac{2\alpha-3}{2(\alpha-1)}}\lambda^{-\frac{1}{2}} + O_{\psi}(\lambda^{-\frac{3}{2}})\\
    &= C^\prime e^{i\lambda \Phi_{\lambda}(\xi_0)}|\xi_0|^{-i\beta\mu} + O_\psi(\lambda^{-\frac{3}{2}}), 
\end{split}
\end{equation*}
as $\lambda \rightarrow \infty$, or equivalently, as $y \rightarrow 0+$. On the other hand, observe that if $\eta \in D_1 \setminus [(\frac{4}{3})^{\frac{1}{\alpha-1}}\xi_0,(\frac{4}{5})^{\frac{1}{\alpha-1}}\xi_0]$, then $|\alpha \eta^{\alpha-1}-y| \geq \frac{1}{4}\eta^{\alpha-1}$. Integrate by parts to obtain
\begin{equation*}
\begin{split}
    |B| &\lesssim \int_{D_1} \left|\partial_\eta\Big(\frac{\eta^{-\beta(1+i\mu)}(1-\zeta(\eta))}{\alpha \eta^{\alpha-1}-y}\Big)\right|d\eta + O(y^{\frac{\alpha}{2(1-\alpha)}})\\
    &\lesssim\int_{D_1 \setminus [(\frac{4}{3})^{\frac{1}{\alpha-1}}\xi_0,(\frac{4}{5})^{\frac{1}{\alpha-1}}\xi_0]} \eta^{-(1+\frac{\alpha}{2})} + \eta^{-\frac{\alpha}{2}}\xi_0^{-1} d\eta + O(y^{\frac{\alpha}{2(1-\alpha)}}) \lesssim y^{\frac{\alpha}{2(1-\alpha)}}.
\end{split}
\end{equation*}
For $y \geq 1$, it suffices to show $\varlimsup\limits_{y \rightarrow \infty} |K(y)|<\infty$. Since $\alpha<1$, the critical point $\xi_0\rightarrow 0$ as $y \rightarrow \infty$. Hence it suffices to show that the integral for $K(y)$ on $\eta \in [1,\infty)$ is uniformly bounded in $y \geq 1$. Since $\partial_\eta(-\frac{\eta^\alpha}{y}+\eta) = -\frac{\alpha}{y}\eta^{\alpha-1} + 1$ is monotonic and is bounded below by $\frac{1}{2}$ for $\eta \geq 2^{-\frac{1}{\alpha-1}}\xi_0$, it follows from the Van der Corput Lemma that
\begin{equation*}
    \left|\int_1^\infty e^{-i\eta^\alpha+iy\eta}\eta^{\beta(1+i\mu)}d\eta\right| \lesssim y^{-1}\int_1^\infty \eta^{-(\beta+1)}d\eta \lesssim y^{-1}.
\end{equation*}
\end{proof}
\begin{proof}[Proof of \cref{nonlin.est.}]
If $F$ is of the form \cref{Schauder} or \cref{Schauder2}, then $F \in C^{k-1,1}_{loc}$ for $k = \lfloor p \rfloor$ or $k = \lfloor p \rfloor-1$, respectively; moreover, if $p$ is an odd integer, then $F$ is a multilinear combination of $u$ and $\overline{u}$, and hence smooth.

For $s=0$, the proof follows from the H\"older's inequality since $F$ is locally Lipshitz and $F(0)=0$. Therefore assume $s>0$. By the Plancherel Theorem,
\begin{equation}\label{lp2}
    \| F(u) \|_{H^s_\alpha} \simeq \| P_{<1} F(u) \|_{L^2} + (\sum_{N \geq 1} N^{2s} \| P_N F(u) \|_{L^2}^2)^{1/2}.
\end{equation}
The low frequency component reduces to $s=0$ case since
\begin{equation*}
    \| P_{<1} F(u) \|_{L^2} \lesssim \| F(u) \|_{L^2} \lesssim  \| u \|_{L^2}.
\end{equation*}
Let $N\geq 1$. Since $\| u \|_{L^\infty}, \| P_{<N} u \|_{L^\infty} =O(1)$ and $F$ is locally Lipschitz, we have a pointwise estimate
\begin{equation*}
    |F(u) - F(P_{<N}u)| \lesssim |P_{ \geq N}u|,
\end{equation*}
and taking $P_N$ both sides, followed by taking the $L^2$ norm,
\begin{equation*}
    \| P_N F(u) \|_{L^2} \lesssim \| P_N F(P_{<N}u)\|_{L^2} + \| P_{\geq N} u \|_{L^2}.
\end{equation*}
Using the Cauchy-Schwarz inequality on the second term,
\begin{equation*}
\begin{split}
    \sum_{N \geq 1} N^{2s} \| P_{\geq N} u \|_{L^2} &\simeq \sum_{N \geq 1} \sum_{N^\prime \geq N} N^{2s}\| P_{N^\prime}u\|_{L^2}^2\\
    &= \sum_{N^\prime \geq 1} \sum_{1 \leq N \leq N^\prime} N^{2s}\| P_{N^\prime} u \|_{L^2}^2 \lesssim \sum_{N^\prime \geq 1} (N^\prime)^{2s} \| P_{N^\prime} u \|_{L^2}^2 \lesssim \| u \|_{H^s_\alpha}^2.
\end{split}
\end{equation*}
Now we wish to show that the first term is summable to a term controlled by $\| u \|_{H^s_\alpha}$. First, assume $s=k$. Then,
\begin{equation*}
\begin{split}
    \| |\nabla_\alpha|^k F(u) \|_{L^2} &= \| (\xi^2+|\eta|^\alpha)^{k/2} \widehat{F(u)}\|_{L^2} \simeq \| (|\xi|^k + |\eta|^{\frac{\alpha}{2}k}) \widehat{F(u)}\|_{L^2}\\
    &\leq \| |\xi|^k \widehat{F(u)} \|_{L^2} + \| |\eta|^{\frac{\alpha}{2}k}\widehat{F(u)}\|_{L^2} \lesssim \| \partial_x^k F(u) \|_{L^2} + \| D^{\frac{\alpha}{2}k}F(u) \|_{L^2}\\
    &\lesssim \| \partial_x^k u \|_{L^2} + \| u \|_{H^{\frac{\alpha}{2}k}} \lesssim \| u \|_{H^k_\alpha}. 
\end{split}
\end{equation*}
By \cref{Bernstein3} and the triangle inequality,
\begin{equation*}
    \| P_N F(P_{<N}u)\|_{L^2} \lesssim N^{-k} \||\nabla_\alpha|^k F(P_{<N}u)\|_{L^2} \lesssim N^{-k}(\| \partial_x^k F(P_{<N}u) \|_{L^2} + \| D_2^{\frac{\alpha}{2}k}F(P_{<N}u)\|_{L^2}).
\end{equation*}
Let $Q_1 = P_{\leq 1}$ and $Q_N = P_N$ if $N>1$. Then by reasoning as \cite[Proposition A.9]{tao2006nonlinear}, we obtain
\begin{equation*}
    \| \partial_x^k F(P_{<N}u) \|_{L^2} \lesssim \sum_{1 \leq N^\prime <N} (N^\prime)^k \| Q_{N^\prime}u \|_{L^2}.
\end{equation*}
For the fractional derivative,
\begin{equation*}
\begin{split}
    \| D_2^{\frac{\alpha}{2}k}F(P_{<N}u)\|_{L^2} &\lesssim \| D^{\frac{\alpha}{2}k}F(P_{<N}u)\|_{L^2} \lesssim \| P_{<N} u \|_{H^{\frac{\alpha}{2}k}}\\
    &\lesssim \sum_{1 \leq N^\prime <N} \| Q_{N^\prime} u \|_{H^{\frac{\alpha}{2}k}} \lesssim \sum_{1\leq N^\prime <N} (N^\prime)^k \| Q_{N^\prime} u \|_{L^2},
\end{split}
\end{equation*}
and therefore combining the $x$ and $y$ derivatives,
\begin{equation*}
\| P_N F(P_{<N}u)\|_{L^2} \lesssim N^{-k} \sum_{1\leq N^\prime <N} (N^\prime)^k \| Q_{N^\prime} u \|_{L^2} \lesssim \Big(\sum_{1 \leq N^\prime <N} (N^\prime)^{2k-\epsilon}N^{-2k+\epsilon}\| Q_{N^\prime} u \|_{L^2}^2\Big)^{1/2}.    
\end{equation*}
By taking $\epsilon>0$ sufficiently small depending on $s,k$, and by \cref{lp2} and the Cauchy-Schwarz inequality,
\begin{equation*}
\Big(\sum\limits_{N \geq 1} N^{2s}\| P_N F(P_{<N}u)\|_{L^2}^2\Big)^{1/2} \lesssim_\epsilon \| u \|_{H^s_\alpha}.   
\end{equation*}
For $F(u) = |u|^{p-1}u$ or $|u|^{p-1}$, it can be verified that the implicit constant is of the form $C\| u \|_{L^\infty}^{p-1}$ or $C \| u \|_{L^\infty}^{p-2}$ where $C=C(s,p,\alpha)>0$.
\end{proof}
\begin{proof}[Proof of \cref{GN inequality}]
Since \cref{gagliardo} is invariant under $u(x,y) \mapsto \mu u(\frac{x}{\lambda},\frac{y}{\lambda^{2/\alpha}})$ for $\mu,\lambda>0$, we assume without loss of generality that $\| u \|_{\dot{W}_\alpha^{s,p}} = \| u \|_{L^p}=1$. By the triangle inequality and \cref{Bernstein},
\begin{equation*}
    \| u \|_{L^q} \leq \sum_{N\in 2^\mathbb{Z}} \| P_N u \|_{L^q} \lesssim \sum_N N^{(1+\frac{2}{\alpha})(\frac{1}{p}-\frac{1}{q})}\| P_N u \|_{L^p} = \sum_N N^{s\theta}\| P_N u \|_{L^p}.
\end{equation*}
Since $\| P_N u \|_{L^p}\lesssim \| u \|_{L^p}=1$ and $\| P_N u \|_{L^p} \simeq N^{-s} \| P_N |\nabla_\alpha|^s u \|_{L^p} \lesssim N^{-s} \| u \|_{\dot{W}_\alpha^{s,p}} = N^{-s}$ by \cref{Bernstein3,Bernstein5},
\begin{equation*}
    \| u \|_{L^q} \lesssim \sum_{N \in 2^\mathbb{Z}} N^{s\theta} \min (1,N^{-s})<\infty.
\end{equation*}
\end{proof}
\bibliographystyle{unsrt}
\bibliography{ref}
\end{document}